\pdfoutput=1
\documentclass[11pt, a4paper]{article} 

\usepackage[english]{babel} 
\usepackage[utf8]{inputenc} 
\usepackage{vmargin} 
\setmargins{2.5cm}       
{1.5cm}                        
{16.5cm}                      
{23.42cm}                    
{10pt}                           
{1cm}                           
{0pt}                             
{2cm}                           

\usepackage{amsmath}
\usepackage{amsthm}
\usepackage{amssymb}
\usepackage{comment}
\usepackage{enumitem} 
\setlist[enumerate,1]  {label={\rm (\roman*)}, leftmargin=1.5em}   
\setlist{noitemsep} 
\usepackage{bigstrut}
\usepackage{float} 
\usepackage{mathrsfs} 
\usepackage{mdframed} 
\usepackage{xcolor} 
\newenvironment{reserva}%
{\begin{mdframed}[backgroundcolor=lime]}
  {\end{mdframed}}
\newenvironment{reserva2}
{\begin{mdframed}[backgroundcolor=pink]}
  {\end{mdframed}}
\usepackage{version}
\excludeversion{reserva} 
\excludeversion{reserva2}
\usepackage[nottoc]{tocbibind}
\usepackage[title]{appendix}
\usepackage{esint}
\usepackage{titlesec}
\usepackage{commath}
\usepackage{mathtools}

 \usepackage[pagebackref]{hyperref}
\hypersetup{
  colorlinks=true,
  citecolor = blue, 
  linkcolor= blue,
}

\numberwithin{equation}{section}
\newtheorem{theorem}{Theorem}[section]
\newtheorem{proposition}[theorem]{Proposition}
\newtheorem{lemma}[theorem]{Lemma}
\newtheorem{corollary}[theorem]{Corollary}
\newtheorem{definition}[theorem]{Definition}

\newtheorem{remark}[theorem]{Remark}

\renewenvironment{proof}[1][Proof] {\noindent \textbf{#1.} }
{\  \rule{0.5em}{0.5em}\par \medskip}

\newcommand{\myeq}[1]{\ensuremath{\stackrel{\text{#1}}{=}}}

\newcommand{\myleq}[1]{\ensuremath{\stackrel{\text{#1}}{\leq}}}
\newcommand{\mygeq}[1]{\ensuremath{\stackrel{\text{#1}}{\geq}}}

\newcommand{\defeq}{\vcentcolon=}

\newcommand{\hole}{\mathscr{C}}

\newcommand\restr[2]{\ensuremath{#1_{|_{#2}}}}

\renewcommand{\abs}[1]{\left| #1 \right|}

\newcommand*\Lap{\mathop{}\!\mathbin\bigtriangleup}

\newcommand{\D}{\displaystyle}

\titleformat{\paragraph}
{\normalfont\normalsize\bfseries}{\theparagraph}{1em}{}
\titlespacing*{\paragraph}
{0pt}{3.25ex plus 1ex minus .2ex}{1.5ex plus .2ex}

\newcommand*\adh[1]{\overline{#1}} 

\newcommand{\RN}{{\mathbb{R}^N}}

\newcommand{\R}{\mathbb{R}}

\def\XXint#1#2#3{{\setbox0=\hbox{$#1{#2#3}{\int}$ }
		\vcenter{\hbox{$#2#3$ }}\kern-.580\wd0}}

\begin{document}

\title{On the loss of mass for the heat equation in an exterior domain
with general boundary conditions}

\author{ Joaquín Domínguez-de-Tena${}^{*,1}$ \\
An\'{\i}bal Rodr\'{\i}guez-Bernal\thanks{Partially supported
    by Projects PID2019-103860GB-I00 and  PID2022-137074NB-I00,  MICINN and   GR58/08
  Grupo 920894, UCM, Spain}\ ${}^{,2}$}

\date{\today}
\maketitle


\setcounter{footnote}{2}
\begin{center}
  Departamento de Análisis Matemático y  Matem\'atica Aplicada\\ Universidad
  Complutense de Madrid\\ 28040 Madrid, Spain \\ and \\
  Instituto de Ciencias Matem\'aticas \\
CSIC-UAM-UC3M-UCM\footnote{Partially supported by ICMAT Severo Ochoa
Grant CEX2019-000904-S funded by MCIN/AEI/ 10.13039/501100011033} , Spain 
\end{center}

\makeatletter
\begin{center}
${}^{1}${E-mail:
joadomin@ucm.es}
\\ 
${}^{2}${E-mail:
arober@ucm.es}
\end{center}
\makeatother

  \noindent {$\phantom{ai}$ {\bf Key words and phrases:}  Heat equation,
    exterior domain, mass loss, asymptotic mass, decay rates, Dirichlet, Neumann,
    Robin, boundary conditions.} 
  \newline{$\phantom{ai}$ {\bf Mathematical Subject Classification
      2020:} \
    35B30, 25B40, 35E15, 35K05}

\begin{abstract}
  In this work, we study the decay of mass for solutions to the heat
  equation in exterior domains, i.e., domains which are the complement
  of a compact set in $\mathbb{R}^N$. Different homogeneous boundary
  conditions are considered, including Dirichlet, Robin, and Neumann
  conditions. We determine the exact amount of mass loss and identify
  criteria for complete mass decay, in which 
  the dimension of the space plays a key role. Furthermore, the paper
  provides explicit mass decay rates. 
\end{abstract}

\section{Introduction}
One of the main  properties of the heat equation in the entire
space
\begin{displaymath}
  		u_t-\Lap u = 0, \qquad x\in \R^{N}, \quad t>0, 
\end{displaymath}
is that the mass of the solution, which is defined as 
\begin{equation*}
	m(t)\defeq \int_{\mathbb{R}^N} u(x,t) dx,
\end{equation*}
is conserved during the temporal evolution. This can be obtained,
formally, by integrating the equation in $\R^{N}$ and assuming the
decay at infinity of the solution or,  more rigorously, by using the
integral representation of the solution using the Gaussian heat
kernel.
Mass conservation  is consistent with
various probabilistic and physical interpretations of the equation and
reflects the phenomenon of diffusion of $u$ in $\R^{N}$. Of course,
for nonnegative solutions, this property implies the conservation of
the $L^{1}(\R^{N})$ norm of the solutions with time.

In bounded domains, the situation changes. For example, if we
consider the heat equation in a bounded domain
$\Omega\subset\mathbb{R}^N$ with homogeneous Dirichlet conditions on
$\partial\Omega$, 
\begin{displaymath}
	\left\{
	\begin{aligned}
		u_t-\Lap u = 0 \quad & in \ \Omega\times(0,\infty) \\
		u=0 \quad & on \ \partial\Omega\times[0,\infty) \\
		u=u_0 \quad & in \ \Omega\times\{0\} , 
	\end{aligned}	
	\right. 
\end{displaymath}
we find that the solutions decay to zero in the  $L^1(\Omega)$
and $L^\infty(\Omega)$ norms and, consequently, the mass decays to $0$
for all solutions. The same occurs if we impose Robin boundary
conditions of the form $\frac{\partial u}{\partial 
  n}+b u=0$ with $b >0$. On the other hand, if we impose homogeneous Neumann
conditions, i.e. $b=0$, the
mass is conserved during the evolution. The physical reason for this
is that mass is lost through the boundary in the case of the first two
boundary conditions while there is no flux through the boundary in the
latter one.  Mathematically, the explanation stems from the sign of
the first eigenvalue of the Laplacian, which is positive in the first
two cases and is zero in the latter.

However, if the domain is unbounded but has nonempty boundary, we
expect to have a flux of mass through the boundary and the Laplacian
to have continuous spectrum $(0,\infty)$  so the evolution of the mass  is unclear.  Actually,
integrating the equation in $\Omega$ we obtain 
\begin{displaymath}
	\frac{d}{dt}\int_\Omega u(x,t)dx=\int_\Omega \Lap u(x,t)dx
        =\int_{\partial \Omega} \frac{\partial u}{\partial n}(x,t)dx
        . 
\end{displaymath}
So, if $u\geq 0$ and  $\restr{u}{\partial \Omega}=0$ then
$\frac{\partial u}{\partial n} \leq 0$ on $\partial \Omega$ and then
$\D \int_\Omega u(x,t)dx$ decreases in
time although we have no quantitative estimate of the decay. 
The same argument holds for Robin boundary conditions $\frac{\partial u}{\partial 
  n}+b  u=0$ with $b >0$, while for Neumann, $b=0$, again
mass is conserved.

In this paper  we consider a connected exterior domain,
that is, the complement of a compact set $\hole$ that we denote the
\emph{hole}, which is the closure of a bounded smooth set; hence, 
$\Omega=\RN\backslash \hole$.  
We will assume $0\in
\mathring{\hole}$, the interior of the hole,  and observe that   $\hole$ may have different connected
components, although $\Omega$ is connected.

As we have shown, the phenomenon of loss
of mass depends on the boundary conditions in the hole and we are interested in
understanding and determining the amount of mass lost for any given
solution. We will show then that the answer depends on  the dimension. If $N\geq 3$, then there will be a
certain remaining mass, while in other cases, all the mass will be
lost through the hole. Also, we will show that we can explicitly
compute the amount of mass lost for each initial data $u_{0} \in
L^{1}(\Omega)$. 
More precisely we will show that there exists a nonnegative function,
$\Phi$, that we denote the asymptotic profile, determined
by the domain and boundary conditions alone, such that the amount of
mass not lost through the hole by a solution with initial data $u_{0}\in
L^{1}(\Omega)$, that is, the \emph{asymptotic mass} of the solution, 
is given by
\begin{displaymath}
  m_{u_0} = \int_\Omega u_0(x) \Phi (x) \,
  dx 
\end{displaymath}
see Proposition \ref{prop:parabolicremaining}. It is then crucial to
understand this function $\Phi$. In this direction we will show that
$\Phi \equiv 1$ for Neumman boundary conditions in any dimensions
(hence no loss of mass at all for any solution), while for Robin or
Dirichlet boundary conditions, if $N\leq 2$ then $\Phi =0$. That is,
all mass is lost through the boundary. On the other hand, if $N \geq
3$, then
\begin{displaymath}
  1-\frac{C}{\abs{x}^{N-2}}\leq \Phi(x) \leq 1 \qquad 
  x\in \Omega 
\end{displaymath}
see Theorem \ref{thm:compperf}. Also, the dependence of the loss of
mass with respect to the boundary conditions is analysed in
Proposition \ref{prop:compasym}.

Finally in Theorem \ref{thm:rate_loss_mass} we study the 
rate of mass loss and prove that, except for Neumann
boundary conditions,  for $N\geq 3$ all
solutions lose mass at a uniform rate, while if  $N\leq 2$ there are
solutions for which the mass  decays to zero  as slow as we want. 
  
The paper is organised as follows.  In Section
  \ref{sec:pre}, we introduce the setting of the problem and the
  general boundary conditions we consider. We  prove the main results
  regarding the existence and regularity of solutions and some
  comparison results that will be very useful thereafter. 
  In Section   \ref{sec:profiles}, we construct the asymptotic profile
  for the problem, which is determined by the domain and the boundary
  conditions. In Section \ref{sec:asymptoticmass} we show that the
  asymptotic profile allows us to explicitly determine the amount of
  mass lost by each solution, see Proposition
  \ref{prop:parabolicremaining}. We will also provide some estimates
  on the behavior of the profile that, in particular, will imply the
  dimension dependent behaviour discussed above, see Theorem
  \ref{thm:compperf}. Appendix \ref{sec:schauder-estimates} contain
  some classical Schauder-type estimates for harmonic functions  that
  are used for the main result.

\section{The problem and preliminary elliptic and parabolic  results}
\label{sec:pre}


In this section we consider a slightly more general setting than that
of an exterior domain, by allowing $\Omega$ to be a connected open set
with compact boundary. That incudes the case of exterior domains but
also bounded ones. 

Hence, we will study the heat equation 
\begin{equation} \label{eq:heat_theta} 
  \left\{
	\begin{aligned}
		u_t-\Lap u = 0 \quad & in \ \Omega\times(0,T) \\
		B_\theta(u)=0 \quad & on \ \partial\Omega\times[0,T] \\
		u=u_0 \quad & in \ \Omega\times\{0\} , 
	\end{aligned}	
	\right. 
\end{equation}
where $u_0\in L^1(\Omega)$ and   we  consider Dirichlet, Robin or
Neumann homogeneous boundary conditions on $\partial \Omega$,  written in the form 
\begin{equation}
	\label{eqn:thetabc}
	B_\theta(u)\defeq \sin(\frac{\pi}{2}\theta(x))\frac{\partial u}{\partial n}+\cos(\frac{\pi}{2}\theta(x))u,
\end{equation}
where $\theta\in C(\partial \Omega, [0,1])$ satisfies one of the
following cases in each connected component of $\partial \Omega$: 
\begin{enumerate}
	\item Dirichlet conditions: $\theta\equiv 0$
	\item Mixed Neumann and Robin conditions:
          $0<\theta\leq 1$ . 
\end{enumerate}
In particular, if $\theta\equiv 1$ we recover Neumann
boundary conditions. 
In general, we will refer to these as homogeneous $\theta$-boundary
conditions.    Note that, by suitably choosing $\theta(x)$,
(\ref{eqn:thetabc}) includes all boundary conditions of the 
form $\frac{\partial u}{\partial n}+b(x)u=0$. The restriction $0\leq
\theta \leq 1$ makes $b(x)\geq 0$ which is the standard dissipative
condition. The reason for these notations will be seen in the results
below about monotonicity of solutions with respect to $\theta$, see
Section \ref{sec:comparison_theta}.

As a general notation, for a given function $\theta$ as above, we  define the Dirichlet part of $\partial \Omega$ as
	\begin{displaymath}
		\partial^D \Omega\defeq \{x \in \partial \Omega \ : \ \theta(x)=0\},
	\end{displaymath}
	 the Robin part of $\partial \Omega$ as
	\begin{displaymath}
		\partial^R \Omega\defeq \{x \in \partial \Omega \ : \ 0<\theta(x)<1\},
	\end{displaymath}
	and  the Neumann part of $\partial \Omega$ as
	\begin{displaymath}
		\partial^N \Omega\defeq \{x \in \partial \Omega \ : \
                \theta(x)=1\} . 
	\end{displaymath}	
The conditions imposed on $\theta$ imply that $\partial^D \Omega$ is a
union of connected components of $\partial \Omega$, although Neumann
and Robin conditions can coexist in the same connected component of
$\partial \Omega$.

In general we will use a superscript $\theta$ to denote anything
related to (\ref{eq:heat_theta}).  For example, the semigroup of
solutions to (\ref{eq:heat_theta}) will be denoted by  $S^\theta(t)$
and the associated kernel by $k^\theta(x,y,t)$, see Section \ref{sec:semigroup_theta}. Sometimes, we
will add as subscript $\Omega$ to indicate the dependence of these
objects in the domain.

\subsection{Some elliptic results}

We present some elliptic results based on an  $L^2$ framework. For
this we will denote 
    \begin{displaymath}
      H^1_\theta(\Omega)=\{u\in H^1(\Omega) \ : \ \restr{u}{\partial^D \Omega}\equiv 0\}.
    \end{displaymath}
which is a closed subspace of
$H^1(\Omega)$. Then we have the following standard result, based on
Lax-Milgram theorem and the coercivity of the bilinear form in $
H^1_\theta(\Omega)$ 
\begin{displaymath}
  		a_\theta(u,\varphi) = \int_{\Omega}\nabla u \nabla
                \varphi + \gamma\int_{\Omega}u\varphi  +
                \int_{\partial^R \Omega}\cot(\frac{\pi}{2}\theta) u
                \varphi . 
\end{displaymath}
Notice that here we use the fact that $0<\theta <1$ on $\partial^R
\Omega$ so $0<\cot(\frac{\pi}{2}\theta)<\infty$ on that set.

\begin{theorem}
	\label{thm:L2ex}
	Given $\Omega$ a domain with compact boundary and some
        homogeneous $\theta$-boundary condition and  $L\in (H^{1}_{\theta}(\Omega))'$, assume  $\gamma>
        0$ or $\gamma =0$ and $\theta \not \equiv 1$, that is, except
        Neumann boundary conditions.  Then the problem 
	\begin{equation}
		\label{eqn:probell}
		\left\{
		\begin{aligned}
			-\Lap u+\gamma u = L \quad & in \ \Omega \\
			B_\theta(u)=0 \quad & on \ \partial\Omega , 
		\end{aligned}	
		\right. 
	\end{equation}	
	has a unique weak solution $u\in H^1_\theta(\Omega)$, that is,
	\begin{equation}
		\label{eqn:thmL2exweak}
		\int_{\Omega}\nabla u \nabla \varphi +
                \gamma\int_{\Omega}u\varphi  + \int_{\partial^R
                  \Omega}\cot(\frac{\pi}{2}\theta) u \varphi=
                L(\varphi)  \qquad \forall \varphi \in
                H^1_\theta(\Omega) 
	\end{equation}
and there exists a constant $C>0$ such that $\norm{u}_{H^1(\Omega)}\leq C\norm{L}_{(H^{1}_{\theta}(\Omega))'}$. 

  In particular, the mapping $(H^{1}_{\theta}(\Omega))' \ni L\mapsto u  \in H^{1}_{\theta}(\Omega)$ is an isomorphism. 
\end{theorem}
\begin{reserva2}
  \begin{proof}
    If we consider
    \begin{equation}
      a_\theta(u,\varphi) = \int_{\Omega}\nabla u \nabla \varphi + \gamma\int_{\Omega}u\varphi  + \int_{\partial^R \Omega}\cot(\frac{\pi}{2}\theta) u \varphi,
    \end{equation}
    we have that it is a bilinear form in $H^1(\Omega)$ (so in
    $H^1_\theta(\Omega)$) which is coercive because, as
    $\cot(\frac{\pi}{2}\theta)u^2\geq 0$
    \begin{equation}
      a_\theta(u,u)\geq \int_{\Omega}\abs{\nabla u}^2 + \gamma\int_{\Omega}u^2 \geq \min(1,\gamma)\norm{u}_{H^1(\Omega)}^2.
    \end{equation}
    Hence, we can use Lax-Milgram theorem (See \cite{evans} Section
    6.2) to obtain a solution of $a_\theta(u,\varphi)=L_f(\varphi)$
    where $L_f(\varphi)=\int_\Omega f\varphi$. In addition, we obtain
    \begin{equation}
      \label{eqn:fach1}
      \norm{u}_{H^1(\Omega)}\leq C\norm{f}_{L^2(\Omega)}.
    \end{equation}
  \end{proof}
\end{reserva2}

When $L \in  (H^{1}_{\theta}(\Omega))'$ is given by a function   $f\in
L^2(\Omega)$ in the sense that $L(\varphi)=  \int_{\Omega} f\varphi$,
with  some modified arguments of the standard theory of regularity we obtain:
\begin{theorem}
	\label{thm:L2re}
	Let $u$ be a weak solution of problem (\ref{eqn:probell}) with
        $\gamma>0$ or $\gamma =0$ and $\theta \not \equiv 1$, that is, except
        Neumann boundary conditions and $L(\varphi)=  \int_{\Omega}
        f\varphi$ with $f\in L^2(\Omega)$.
        If the boundary $\partial \Omega$ is of class $C^2$ and $\theta\in
        C^1(\partial \Omega)$, then $u\in H^2(\Omega)$ and there exists a
        $C>0$ independent of $f$ such that 
	\begin{equation}
		\label{eqn:L2reeq1}
		\norm{u}_{H^2(\Omega)}\leq C\norm{f}_{L^2(\Omega)}.	
              \end{equation}

In particular, the mapping $f\mapsto u$ defines an
isomorphism from $L^{2}(\Omega)$ into
\begin{displaymath}
 	D(\Lap_\theta)  = \{u \in H^{2}(\Omega), \ B_{\theta}(u) =0 \
        \mbox{on $\partial\Omega$} \} 
\end{displaymath}
endowed with the norm of $H^{2}(\Omega)$, which is dense in
$H^{1}_{\theta}(\Omega)$.   The inverse of this 
operator is $-\Delta + \gamma I$ on   $D(\Lap_\theta)$ and is
a closed and selfadjoint  operator in  $L^{2}(\Omega)$.  
\end{theorem}
\begin{proof}
If $\Omega$ is bounded the result is standard and can be found for
Dirichlet boundary conditions e.g. in \cite{evans} Section 6.3 Theorem 4 and for Neumann and
Robin ones in \cite{Mikhailov} Page 217,
Theorem 4 and footnote.
If $\Omega$ is
unbounded, and hence an exterior domain,  
the proof is not easy to find in the literature, so we give a proof. 

In this case,   consider $\Omega_R=\Omega\cap B(0,R)$ with $R$ large enough so
  that $\partial \Omega \subset \partial \Omega_R$. Then, we consider
  a cut-off function $\chi\in C^\infty_c(\Omega_{2R})$ such that
  $\chi(\Omega_R)\equiv 1$. Then  $\tilde{u} = u\chi$ is a weak
  solution of 
	\begin{displaymath}
		\left\{
		\begin{aligned}
			- & \Lap \tilde{u} + \gamma \tilde{u} = \tilde{f} &&in \ \Omega_{2R} \\
			& B_\theta(\tilde{u})=0 &&on \ \partial \Omega_{2R}
		\end{aligned}
		\right.
	\end{displaymath}
with $\tilde{f}= f\chi -\nabla u \nabla \chi - u \Lap \chi$. 
Hence,
\begin{displaymath}
\norm{\tilde{f}}_{L^2(\Omega_{2R})}\leq C\norm{f}_{L^2(\Omega_{2R})} +
C\norm{u}_{H^1(\Omega_{2R})}\myleq{Thm \ref{thm:L2ex}}
C\norm{f}_{L^2(\Omega_{2R})}. 
\end{displaymath}

Now, from the regularity of the boundary  $\partial \Omega$ and
$\theta$, we can use classical elliptic
regularity results in  bounded domains (see for example
\cite{Mikhailov} Page 217,  Theorem 4 and footnote) to obtain
$\norm{\tilde{u}}_{H^2(\Omega_{2R})}\leq
C\norm{\tilde{f}}_{L^2(\Omega_{2R})} \leq
C\norm{f}_{L^2(\Omega_{2R})}$.  
Thus,
\begin{equation}
	\label{eqn:uacofcerca}
	\norm{u}_{H^2(\Omega_{R})}\leq C \norm{f}_{L^2(\Omega)}.
\end{equation}

Now, classical interior regularity results (see for example
\cite{evans} Section 6.3) guarantee that $u\in
H^2_{loc}(\Omega)$ and, for any two concentric balls
$B(x,r)\subset B(x,2r)\subset\Omega$ 
\begin{displaymath}
	\label{eqn:reg1}
	\norm{u}_{H^2(B(x,r))}\leq C(r)(\norm{f}_{L^2(B(x,2r))}+\norm{u}_{L^2(B(x,2r))}).
\end{displaymath}
Then, choosing $r>0$ sufficiently small and covering
$\Omega\backslash \Omega_{R}$ with a countable family of balls
$B(x_i,r)\subset B(x_i,2r)\subset \Omega)$ in a way that every
point $x\in\Omega$ is contained only in a finite number  ($m$,
independent of $x$) of balls $B(x_i,2r)$ we obtain 
\begin{equation} \label{eqn:uacoflejos}
\begin{aligned}
\norm{u}_{H^2(\Omega\backslash\Omega_R)} &
\leq \sum_{i} \norm{u}_{H^2(B(x_i,r))}\leq
\sum_{i}  C(r)(\norm{f}_{L^2(B(x_i,2r))}+\norm{u}_{L^2(B(x_i,2r))})
\\
& \leq m
C(r)(\norm{f}_{L^2(\Omega)}+\norm{u}_{L^2(\Omega)})
\myleq{Thm \ref{thm:L2ex}}
C\norm{f}_{L^2(\Omega)}. 
\end{aligned}
\end{equation}
Finally, combining (\ref{eqn:uacofcerca}) and (\ref{eqn:uacoflejos}), we obtain (\ref{eqn:L2reeq1}).

Once $u\in H^{2}(\Omega)$, integrating in parts in the weak
formulation (\ref{eqn:thmL2exweak}) we easily get $B_{\theta}(u)=0$ on
$\partial \Omega$. Hence, the description of $D(\Delta_{\theta})$
follows. That this space is dense in $H^{1}_{\theta}(\Omega)$ is
because  the isomorphisms in Theorems \ref{thm:L2ex} and
\ref{thm:L2re} and the fact that $H^{1}_{\theta}(\Omega)$ is dense in
$L^{2}(\Omega)$, which implies in turn that $L^{2}(\Omega)$ is dense
in $ (H^{1}_{\theta}(\Omega))'$. 
The rest also follows easily. 
      \end{proof}

\begin{reserva2}
 
  \begin{proposition}
    Given $\gamma>0$, $(\Lap_\theta+\gamma, D(\Lap_\theta))$ is a
    closed self-adjoint bijective operator.
  \end{proposition}
  \begin{proof}
    The bijectivity is a consequence of Theorems \ref{thm:L2re} and \ref{thm:L2ex}.\\
	
    Let us proved the closeness. Consider $u_n\in D(\Lap_\theta)$ such
    that $\norm{u_n-u}_{L^2(\Omega)}\to 0$ and
    $\norm{\Lap_\theta u_n - f}_{L^2(\Omega)}\to 0$. Then, rename
    $f_n=\Lap_\theta u_n$. As $u_n-u_m$ satisfy $B_\theta(u_n-u_m)=0$
    in $\partial \Omega$ and $\Lap_\theta (u_n-u_m) = f_n-f_m$ in
    $\Omega$, we have, using Theorem \ref{thm:L2re}:
    \begin{equation}
      \norm{u_n-u_m}_{H^2(\Omega)}\leq C\norm{f_n-f_m}_{L^2(\Omega)}\to 0.
    \end{equation}
    Therefore, $\{u_n\}$ is a Cauchy sequence in $H^2(\Omega)$ so it
    converges $u_n\to \tilde{u}$ in $H^2(\Omega)$. As $u_n\to u$ in
    $L^2(\Omega)$, we have that $\tilde{u}=u$, so $u\in H^2(\Omega)$,
    $B_\theta(u)=0$ and $\Lap_\theta u = f$. 
	
    Let us prove that $\Lap_\theta$ is self-adjoint. Take
    $u,v\in D(\Lap_\theta)$. Then,
    \begin{equation}
      \int_\Omega u\Lap_\theta v = \int_\Omega \Lap_\theta u v +
      \int_\Omega \left(u\frac{\partial v}{\partial n}-v\frac{\partial
          u}{\partial n}\right) = \int_\Omega \Lap_\theta u v +
      \int_{\partial^R \Omega} \cot(\frac{\pi}{2}\theta)
      \left(uv-vu\right) = \int_\Omega \Lap_\theta u v. 
    \end{equation}
  \end{proof}
\end{reserva2}

\subsection{Semigroup generated by $\Lap_\theta$}
\label{sec:semigroup_theta}

Now we present some results about the semigroup of solutions
associated to (\ref{eq:heat_theta}), that we will denote
$S^\theta(t)$, which is the semigroup generated by $\Lap_\theta$. If
at some point we want to stress the dependence on 
the domain, we will denote it by $S_\Omega^\theta(t)$. 

We start with the case of initial data in $L^{2}(\Omega)$. 
\begin{theorem}
  \label{thm:semigrupo_L2}
  
  Given $\Omega$ a domain with compact boundary and some
        homogeneous $\theta$-boundary conditions, the operator $(\Lap_\theta,
        D(\Lap_\theta))$ generates an analytic  $C^0$ semigroup of contractions
        $\{S^\theta(t)\}_{t>0}$ in $L^2(\Omega)$, that is,
a family  of bounded linear functions from   $L^2(\Omega)$ into itself
such that:  
\begin{enumerate}
\item
  Semigroup property:  $S^\theta(t+s)u_0=S^\theta(t)S^\theta(s)u_0$ for
         every $0<s<t$ and $u_{0}\in L^{2}(\Omega)$.
         
\item
         $C^0$ property: $\lim_{t\to 0}
         S^\theta(t)u_0 = u_0$ in $L^2(\Omega)$ for every
         $u_0\in L^2(\Omega)$.
         
\item
         Contraction property: $\norm{S^\theta(t)}_{\mathcal{L}(L^2(\Omega))}\leq 1$ for every $t>0$.

\item
         Satisfies the PDE: The semigroup is  analytic and therefore for every $u_{0}\in
         L^{2}(\Omega)$,          $u(t)=S^\theta(t)u_0 \in D(\Delta_{\theta})$ for $t>0$ and  satisfies
	\begin{displaymath}
		\left\{
		\begin{aligned}
& \frac{d}{d t} u(t)-\Lap_\theta u(t) = 0 && \forall t>0 \\
& B_\theta(u(t))=0 &&\forall t> 0 . 
\end{aligned}
			\right. 
		\end{displaymath}

        \item
Assume furthermore that the boundary $\partial \Omega$ is of class
$C^{m}$ and $\theta\in C^{m}(\partial \Omega)$ for $m$ large enough. Then for  $u_0\in
L^2(\Omega)$,  $u (x,t)=  S^\theta(t)u_0(x)$ is a $C^{2,1}(\adh{\Omega}\times(0,\infty))$
solution of the heat equation, that is 
		\begin{displaymath}
			\left\{
			\begin{aligned}
				u_t(x,t)-\Lap u(x,t) & = 0 && \forall(x,t)\in\Omega\times(0,\infty)
                                \\
				B_\theta(u)(x,t) & =0 && \forall x\in
                                \partial \Omega, \ \forall t>0 . 
			\end{aligned}
		\right. 
		\end{displaymath} 
              \end{enumerate}
\end{theorem}
\begin{proof}
(i)-(iv). 	Let us see that the operator $\Lap_{\theta}$ is dissipative. Take $u\in
        D(\Lap_\theta)$, then using $\cot(\frac{\pi}{2}\theta)u\geq
        0$, 
	\begin{displaymath} 
		\int_\Omega u\Lap u =-\int_\Omega \abs{\nabla u} ^2 +
                \int_{\partial \Omega} u \frac{\partial u}{\partial
                  n} =  -\int_\Omega \abs{\nabla u} ^2 -
                \int_{\partial^{R}\Omega} \cot(\frac{\pi}{2}\theta)u^2 \leq   0 .
	\end{displaymath}
	
	In addition, $D(\Lap_\theta)$ is dense in $L^2(\Omega)$ and
        $(0,\infty) \subset \rho(\Lap_\theta)$,  due to
        Theorem \ref{thm:L2ex} and Theorem \ref{thm:L2re}. Thus, we
        can use Lumer-Phillips Theorem (See \cite{pazy} Chapter 1
        Theorem 4.3) to obtain the $C^{0}$ semigroup of contractions. 
        The analiticity is a consequence of the selfadjointness of the
        operator see Theorem 3.2.1 in \cite{caz90:_introd_prob_evol}. 

        Finally, for (v), since  the semigroup is analytic, 
$S(t): L^2(\Omega) \to  D((-\Lap_\theta)^{k}) $ 
is continuous for any $k\in \mathbb{N}$. Now, using higher regularity estimates up
to the boundary (See \cite{evans} Section 6.3,
\cite{gilbarg2015elliptic} Section 6 or \cite{Mikhailov} Section
IV.2), we have that, if the boundary $\partial \Omega$ is of class
$C^{2k}$ and $\theta\in C^{2k-1}(\partial \Omega)$, then
$D((-\Lap_\theta)^{k})\subset 
H^{2k}(\Omega)$ continuously. Hence, for sufficiently large $k$,
\begin{reserva2}
	El resultado en concreto de Mikahilov se encuentra en P 217 Teorema 4.
\end{reserva2}
we have that   $S(t): L^2(\Omega) \to C^{2}(\adh{\Omega})$ is
continuous.
\begin{reserva}
  Esto requiere $2k-\frac{N}{2}> 2$, es decir $k>1+\frac{N}{4}$ 
\end{reserva}
This and Lemma 3.1 of  \cite{ACDRB04_linear}
implies  that, for any $u_0\in L^2(\Omega)$, $t\mapsto S(t)u_{0} \in C^2(\adh{\Omega})$ is
analytic. Therefore, $u(x,t) = S(t)u_0(x)$ belongs to $C^{2,1}(\adh{\Omega}\times(0,\infty))$.
\begin{reserva2}
  The same holds for $u_0\in L^1(\Omega)$ because
  $S(t/2)u_0\in L^2(\Omega)$ and the semigroup property
  $S(t)=S(t/2)S(t/2)u_0$.
\end{reserva2}
      \end{proof}

Now we show that we can extend the semigroup above to $L^{p}(\Omega)$
spaces  and it has nice properties. 

\begin{theorem}
	\label{thm:propiedades_sg_theta}
	The semigroup $\{S^\theta(t)\}_{t>0}$ above has the following properties:
	\begin{enumerate}
        \item
          It  extends to a semigroup of contractions in $L^p(\Omega)$
          for $1\leq p\leq\infty$ which is $C^0$ if $p\neq \infty$ and
          analytic if $1<p<\infty$. 

        \item
          $S^\theta(t)u_0\geq0$ for every $0\leq u_0\in L^p(\Omega)$
          with $1\leq p \leq \infty$. That is, the semigroup is order
          preserving.
          
        \item
          $S^\theta(t)$ is selfadjoint in $L^2(\Omega)$ and moreover
          for $1\leq p\leq\infty$, 
          \begin{equation}
			\label{eqn:Sexchange}
\int_\Omega fS^\theta(t)g = \int_\Omega gS^\theta(t)f  \qquad
\mbox{for all} \  f\in L^p(\Omega), \   g\in L^q(\Omega) 
		\end{equation}
	where $q$ is the conjugate of $p$,  that is $\frac{1}{p}+\frac{1}{q}=1$.

      \item
        If $\partial \Omega$ and $\theta$ are regular enough, the semigroup has an integral positive  kernel, that is $k^\theta:
        \Omega \times \Omega \times (0,\infty) \to
        (0,\infty)$ such that for all $1\leq p\leq\infty$ and
        $u_{0}\in L^{p}(\Omega)$,  
	\begin{equation}
		\label{eqn:ackrnpre}
		S^\theta(t)u_0(x)=\int_\Omega k^\theta(x,y,t)u_0(y)dy
                , \qquad x\in \Omega, \quad t>0. 
	\end{equation}
Moreover $k^\theta(x,y,t)=k^\theta(y,x,t)$. 
        
      \item
        $\{S^\theta(t)\}_{t>0}$ is an analytic $C^0$ semigroup in
        $BUC_\theta(\Omega)=\{u\in BUC(\Omega) : \restr{u}{\partial^D
          \Omega}\equiv 0\}$ where  $BUC(\Omega)$ is the space of  bounded uniformly
        continuous functions.
	\end{enumerate}
\end{theorem}
        \begin{proof}
For simplicity in the proof we will not write the superscript
$\theta$. 

        \noindent (i)-(ii)-(iii). 
Observe that from Theorem \ref{thm:L2re} the  the quadratic form
associated to $-\Lap_{\theta}$ is given by 
\begin{displaymath}
 Q(u)= \int_\Omega \abs{\nabla u}^2 + \int_{\partial \Omega^{R}}
 \cot(\frac{\pi}{2}\theta)u^2, \qquad u \in H^{1}_{\theta}(\Omega) . 
\end{displaymath}
\begin{reserva}
  Es cerrada en el sentido de Davies, pagina 7, y es la FQ asociada
  $-\Delta_{\theta}$ ya que la FB de Lax Milgram y $-\Delta_{\theta}$
  verifican
  \begin{displaymath}
    a(u,v) = < -\Delta_{\theta} u, v>_{L^{2}} 
  \end{displaymath}
  y $D(\Delta_{\theta})$ es denso en $H^{1}_{\theta}$. 
\end{reserva}
\begin{reserva2}
  Now, using \cite{davies1980} Theorem 4.14, we obtain that $Q$ is
  closable and its closure is the form of a positive self-adjoint
  operator which extends $-\Lap$. Let us identify which is the closure
  of $Q$. To do this, we consider its associated norm
  $\norm{\cdot}_{Q}\defeq \norm{\cdot}_{L^2(\Omega)}+Q(\cdot)$. Using
  the trace inequality we have the equivalence of norms
  $\norm{\cdot}_{Q}\cong \norm{\cdot}_{H^1(\Omega)}$. Now, note that
  the closure of $D_\theta(\Lap)$ in $L^2(\Omega)$ with respect to the
  norm $H^1(\Omega)$ is $H^1_\theta(\Omega)$. Hence, we can extend $Q$
  to $H^1_\theta(\Omega)\times H^1_\theta(\Omega)$ so that we obtain
  the closure of $Q$ (and we will denote it the same way).

  Now we want to apply \cite{davies1989heat} Lemma 1.3.4 to $-\Lap$
  and $Q$. If we take $f\in H^1_\theta(\Omega)$, then
  $\abs{f}\in H^1_\theta(\Omega)$ and
  $\norm{\abs{f}}_{H^1(\Omega)}\leq \norm{f}_{H^1(\Omega)}$ (See for
  example \cite{davies1989heat} Lemma 1.2.9). Therefore,
  \begin{displaymath}
    Q(\abs{f})=\int_\Omega \abs{\nabla\abs{f}}^2 + \int_{\partial
      \Omega^{R}} \cot(\frac{\pi}{2}\theta)\abs{f}^2 \leq Q(f). 
  \end{displaymath}
  In the same way, one can prove that
  $\max(0,\min(f,1))\in H^1_\theta(\Omega)$ and
  $Q(\max(0,\min(f,1)))\leq Q(f)$ (See \cite{davies1989heat} Theorem
  1.3.5 and Theorem 1.3.9). also Lemma 1.3.4
\end{reserva2}

Since for $u\in H^{1}_{\theta}(\Omega)$ we have $|u|\in
H^{1}_{\theta}(\Omega)$, then from Theorem 1.3.2 in
\cite{davies1989heat} we have that semigroup is order preserving in
$L^{2}(\Omega)$.

Also, if  $0\leq u\in H^{1}_{\theta}(\Omega)$ we have  $v=\max\{u, 1\}
  \in H^{1}_{\theta}(\Omega)$  and 
\begin{displaymath}
  Q(v) = \int_\Omega \abs{\nabla v}^2 + \int_{\partial
      \Omega^{R}} \cot(\frac{\pi}{2}\theta)\abs{v}^2 \leq Q(u). 
\end{displaymath}
Therefore, Theorem 1.3.3 in \cite{davies1989heat} implies that we have
an order preserving  semigroup of contractions in $L^{p}(\Omega)$,
$1\leq p\leq \infty$. Additionally, by Theorem 1.4.1 in
\cite{davies1989heat} these semigroups are  
consistent in the sense that they coincide on  $L^p(\Omega)\cap
L^q(\Omega)$ for any $p,q$ and satisfy the duality property  (\ref{eqn:Sexchange}). Finally,
the analyticity dor  $1<p<\infty$ follows from Theorem 1.4.2 in 
\cite{davies1989heat}.

\medskip 
\noindent (iv)
As in Theorem \ref{thm:semigrupo_L2}, since $S(t)$ is an analytic semigroup in $L^2(\Omega)$, for any
$t>0$ and $k\in \mathbb{N}$, if $\partial \Omega$ and $\theta$ are sufficiently regular, $S(t): L^2(\Omega) \to
D((-\Lap_\theta)^{k}) \subset H^{2k}(\Omega)$ continuously. 

Thus, taking $k$ large,
\begin{reserva}
  esto requiere menos regularidad que en el Teorema
  \ref{thm:semigrupo_L2}: $2k-\frac{N}{2}>0$, es decir $k> \frac{N}{4}$ 
\end{reserva}
we obtain that,  for any
$t>0$, $S(t): L^2(\Omega) \to   L^\infty(\Omega)$ is
continuous and by  duality we also have that $S(t): L^1(\Omega)
\to L^2(\Omega)$ is continuous. Therefore, using  $S(t)=S(t/2)S(t/2)$
we have that $S(t): L^1(\Omega) \to L^\infty(\Omega)$ is also
continuous. Hence, we can  
use \cite{arendt} Theorem 4.16,  to prove the existence of the
kernel,  $k^\theta(x,y,t)$, which is  positive because $S(t)$ is order
preserving.

In particular for  $f, g \in C^\infty_c(\Omega)$,
(\ref{eqn:Sexchange}) implies 
	\begin{displaymath}
		\int_{\Omega\times\Omega}
                k^\theta(x,y,t) f(x) g(y)dxdy=
                \int_{\Omega\times\Omega} k^\theta(y,x,t) f(x) g(y)dydx
	\end{displaymath}
and therefore $k^\theta(x,y,t)=k^\theta(y,x,t)$.

\medskip 
\noindent (v)
This follows from  Theorem 2.4 in
\cite{mora1983semilinear}.
\end{proof}

In addition, the positivity of the semigroup gives us a useful
consequence. 
\begin{corollary}
	\label{cor:absvalue}

 With the notations above, for  any  $u_0\in L^p(\Omega)$, $1\leq
 p\leq \infty$, 
 \begin{displaymath}
\abs{S^\theta(t)u_0(x)}\leq S^\theta(t)\abs{u_0}(x), \qquad x\in
     \Omega, \ t>0 . 
   \end{displaymath}
 
\end{corollary}
\begin{proof}
Splitting $u_0=u_0^+-u_0^-$ and using the linearity and positivity of the semigroup:
\begin{displaymath}
\begin{aligned}
\abs{S^\theta(t)u_0(x)} & =
\abs{S^\theta(t)u_0^+(x)-S^\theta(t)u_0^-(x)} \leq
\abs{S^\theta(t)u_0^+(x)} + \abs{S^\theta(t)u_0^-(x)} \\
			& = S^\theta(t) u_0^+(x) + S^\theta(t)u_0^-(x)
                        = S^\theta(t)\abs{u_0}(x) . 
\end{aligned}
\end{displaymath}
\end{proof}

\subsection{Comparison principles with general boundary conditions}
\label{sec:comparison_theta}

Now we prove some monotonicity results for the solutions of the elliptic
and parabolic problems above as well as monotonicity with respect to
the function $\theta$.

\begin{theorem}
\label{thm:neugeqdirell}
Let $\Omega\subset\RN$ be a domain with compact boundary and
let  $f_{1}, f_{2} \in L^{2}(\Omega)$ and
$g_{1}, g_{2} \in L^{2}(\partial\Omega)$. 
 Let  $u_1,
u_2 \in H^1_\theta(\Omega)$ be two  weak 
solutions of: 
\begin{displaymath}
\label{eqn:heatgen2ell}
\left\{
\begin{aligned}
	-\Lap u_i +\gamma u_{i}= f_i \quad & in \ \Omega \\
	B_\theta(u_i)=g_i \quad & in \ \partial\Omega   , 
\end{aligned}	
\right. 
\end{displaymath}
where $i=1,2$, in the sense that 
  $u_{i}= g_{i}$ on $\partial^{D}\Omega$ and 
for any $\varphi\in   H^1_\theta(\Omega)$,
        \begin{displaymath}
          \int_\Omega \nabla u_{i} \nabla \varphi + \gamma
          \int_{\Omega}  u_{i} \varphi + 
 \int_{\partial^R \Omega}   \cot(\frac{\pi}{2}\theta) u_{i} \varphi  =
 \int_{\partial^R \Omega\cup
  \partial^N\Omega}  \frac{g_{i}}{\sin(\frac{\pi}{2}\theta)}   \varphi
+ \int_\Omega f_{i} \varphi . 
        \end{displaymath}
        
        Then, if $f_1\geq f_2$, $g_1\geq g_2$ and $\gamma >0$ or
        $\gamma=0$ but $\theta\not\equiv
        1$, that is, except for Neumann boundary conditions, we have 
	\begin{displaymath}
		u_1\geq u_2 \qquad x\in \Omega . 
	\end{displaymath}
\end{theorem}
\begin{proof}
	We take $v=u_2-u_1$. It satisfies:
        \begin{displaymath}
    \int_\Omega \nabla v \nabla \varphi + \gamma
          \int_{\Omega}  v \varphi + 
 \int_{\partial^R \Omega}   \cot(\frac{\pi}{2}\theta) v \varphi  =
 \int_{\partial^R \Omega\cup
  \partial^N\Omega}  \frac{g }{\sin(\frac{\pi}{2}\theta)}   \varphi
+ \int_\Omega f \varphi .         
\end{displaymath}
with $f=f_2-f_1\leq 0$ and $g=g_2-g_1\leq 0$. Then we take
$0\leq \varphi= v^{+} =\max\{v,0\} \in
H^1_\theta(\Omega)$ to get 
	\begin{displaymath}
		\label{eqn:heatgen2elleq3}
		\int_\Omega\abs{\nabla v^+}^2 + \gamma
                \int_\Omega\abs{v^+}^2+ \int_{\partial^R
                  \Omega}\cot(\frac{\pi}{2}\theta)|v^+|^2 \leq 0 .  
	\end{displaymath}

        If $\gamma >0$ we get $ \int_\Omega\abs{v^+}^2=0$ and then
        $v\leq 0$ in $\Omega$ as claimed. 
  On the other hand,  if  $\gamma =0$ but $\theta \not\equiv 1$, we obtain
	\begin{displaymath}
		\int_\Omega\abs{\nabla v^+}^2 = 0, \qquad 	\int_{\partial^R \Omega}\cot(\frac{\pi}{2}\theta)|v^+|^2=0.
	\end{displaymath}
	Hence,    from the first term above, $v^+$ is constant in
        $\Omega$ and from the second we get $v^+\equiv 0$. Whence
        $v\leq 0$ again. 
\end{proof}
We will also be able to compare solutions with different
types of boundary conditions. The following theorem, as well as its parabolic version (Theorem \ref{thm:neugeqdir2}, presented later), justifies the
parametrization used for the boundary conditions $B_\theta$ in
(\ref{eqn:thetabc}).

\begin{theorem}
\label{thm:neugeqdir2ell}
Let $\Omega\subset\RN$ be a domain with compact boundary   $\partial
\Omega$  of class $C^2$ and $\theta\in C^1(\partial \Omega)$ and let
$0\leq f \in L^{2}(\Omega)$.  Let  $u_i \in D(\Delta_{\theta{i}} )$, $i=1,2$,  be solutions of 
\begin{displaymath}
\label{eqn:heatgen2ell2}
\left\{
\begin{aligned}
-\Lap u_{i} +\gamma u_{i} = f \quad & in \ \Omega  \\
B_{\theta_{i}} (u_i)=0 \quad & on \ \partial\Omega 
\end{aligned}	
\right. 
\end{displaymath}

\begin{reserva}
  in the sense that $u_{i}= 0$ on $\partial_{i}^{D}\Omega$ and for any
  $\varphi\in H^1_{\theta_{i}} (\Omega)$,
  \begin{displaymath}
    \int_\Omega \nabla u_{i} \nabla \varphi + \gamma
    \int_{\Omega}  u_{i} \varphi + 
    \int_{\partial^R_i \Omega}   \cot(\frac{\pi}{2}\theta_{i}) u_{i} \varphi  =
        \int_\Omega f  \varphi . 
  \end{displaymath}
\end{reserva}
with $0\leq \theta_{i}\leq 1$, $\gamma >0$ or $\gamma =0$ but
$\theta_{i} \not \equiv 1$, that is, 
except for Neumann boundary conditions.

Then, 
\begin{displaymath}
  \theta_1\leq \theta_2 \quad \mbox{implies}\quad
  u_1\leq u_2.
\end{displaymath}
\end{theorem}
\begin{proof}
  From Theorem \ref{thm:neugeqdirell}, $u_{i}\geq 0$ and from Theorem
  \ref{thm:L2re} they
  satisfy the boundary conditions pointwise in $\partial \Omega$. 
\begin{reserva}
HERE continuity of $\theta$ is enough. 

  Note that, as $\theta_1\leq \theta_2$,
  $\partial^D_2\subset \partial^D_1$ and
  $\partial^R_1\cup \partial^N_1\subset \partial^R_2\cup
  \partial^N_2$.  Take $v= u_{1}- u_{2}\in H^1_\theta(\Omega)$ which
  satisfies, for every $\varphi \in H^1_\theta(\Omega)$,
    \begin{displaymath}
      \int_\Omega \nabla v \nabla \varphi + \gamma
      \int_{\Omega}  v \varphi + 
      \int_{\partial^R_2 \Omega}   \cot(\frac{\pi}{2}\theta_{2}) v  \varphi  = \int_{\partial^R_1 \Omega}  \Big( \cot(\frac{\pi}{2}\theta_{2})-\cot(\frac{\pi}{2}\theta_{1}) \Big) u_1  \varphi  +\int_{\partial^R_2 \Omega\backslash \partial^R_1 \Omega}  \cot(\frac{\pi}{2}\theta_{2}) u_1  \varphi 
         \end{displaymath}
Note that $\partial^R_2 \Omega\backslash \partial^R_1 \Omega\subset \partial^D_1$, so $u_1\equiv0$ in $\partial^R_2 \Omega\backslash \partial^R_1$ and the last integral term is zero. In addition, as $u_1\leq 0$ and $\theta_1\leq \theta_2$, we have that $\Big( \cot(\frac{\pi}{2}\theta_{2})-\cot(\frac{\pi}{2}\theta_{1}) \Big) u_1\leq 0$. Hence, we can apply Theorem \ref{thm:neugeqdirell} and $v\leq 0$ as we wanted to prove.
  \end{reserva}
  
Now observe that 
 $\sin(\frac{\pi}{2}\theta_1) \leq \sin(\frac{\pi}{2}\theta_2)$,
$\cos(\frac{\pi}{2}\theta_1) \geq \cos(\frac{\pi}{2}\theta_2)$ and
since $u_{2} \geq 0$ then  $\frac{\partial u_2}{\partial n}\leq 0$ on
$\partial \Omega$. This is clear in the Dirichlet and Neumann parts of
the boundary and in the Robin one is because
$0<\theta_{2}<1$. Therefore, 
\begin{displaymath}
B_{\theta_1}(u_2)=\sin(\frac{\pi}{2}\theta_1)\frac{\partial
  u_2}{\partial n}+\cos(\frac{\pi}{2}\theta_1)u_2\geq
\sin(\frac{\pi}{2}\theta_2)\frac{\partial u_2}{\partial
  n}+\cos(\frac{\pi}{2}\theta_2)u_2 = B_{\theta_{2}}(u_{2}) =0 . 
\end{displaymath}

Thus $v=u_{2}-u_{1}$ satisfies $-\Delta v + \gamma v =0$ in $\Omega$
and $B_{\theta_{1}}(v) \geq 0$ in $\partial \Omega$ and then 
 Theorem \ref{thm:neugeqdirell} implies $v\geq 0$. 
\end{proof}

For the parabolic problems, we have the following results. 

\begin{theorem}
\label{thm:neugeqdir}
Let $\Omega\subset\RN$ be a domain with compact boundary and
let ${u_1}_0, {u_2}_0\in L^{2}(\Omega)$, $f_{1}, f_{2} \in L^{1}((0,T), L^{2}(\Omega))$ and
$g_{1}, g_{2} \in L^{1}((0,T), L^{2}(\partial\Omega))$
with $T>0$.  
Finally, assume  $u_1,u_2\in C^1((0,T), H^1_\theta(\Omega))\cap
C([0,T], L^2(\Omega))$ are such that they are weak solutions of the
problems 
\begin{displaymath}
	\label{eqn:heatgen2}
	\left\{
	\begin{aligned}
		\frac{\partial}{\partial t}u_i-\Lap u_{i} = f_i \quad & in \ \Omega\times(0,T) \\
		B_\theta(u_i)=g_i \quad & on \
                \partial\Omega\times (0,T) \\
		u_{i}={u_{i,0}} := u_{i}(0) \quad & in \
                \Omega\times\{0\} , 
	\end{aligned}	
	\right. 
\end{displaymath}
for $i=1,2$,  in the sense that 
  $u_{i}= g_{i}$ on $\partial^{D}\Omega \times (0,T)$ and 
for any $\varphi\in   C([0,T],  H^1_\theta(\Omega))$, 
	\begin{displaymath}
 \int_\Omega (u_{i})_t\varphi  + \int_\Omega \nabla u_{i} \nabla \varphi +
 \int_{\partial^R \Omega}   \cot(\frac{\pi}{2}\theta) u_{i} \varphi  =
 \int_{\partial^R \Omega\cup
  \partial^N\Omega}  \frac{g_{i}}{\sin(\frac{\pi}{2}\theta)}   \varphi + \int_\Omega f_{i} \varphi  \qquad t
 \in (0,T). 
	\end{displaymath}

        Then, if $f_1\geq f_2$, $g_1\geq g_2$ and ${u_{1,0}}\geq
        {u_{2,0}}$, we have 
	\begin{displaymath}
		u_1\geq u_2  \qquad x\in \Omega, \ t \in (0,T) . 
	\end{displaymath}
\end{theorem}
\begin{proof}
The function  $v\defeq u_2-u_1$  satisfies $v(0) = u_{2,0}- u_{1,0} \leq
0$ and  for  any $\varphi\in   C([0,T], H^1_\theta(\Omega))$, 
\begin{displaymath}
\int_\Omega v_t\varphi + \int_\Omega \nabla v\nabla
\varphi +  \int_{\partial^R \Omega}   \cot(\frac{\pi}{2}\theta) v \varphi =  \int_{\partial^R \Omega\cup
  \partial^N\Omega}  \frac{g}{\sin(\frac{\pi}{2}\theta)} \varphi+\int_\Omega f\varphi 
\end{displaymath}
with $f= f_2-f_1\leq 0$ and $g= g_2-g_1\leq 0$. 
              
	Now take $0\leq \varphi =v_{+}=\max\{v,0\} \in
        C([0,T], H^1_\theta(\Omega))$ to obtain 
	\begin{displaymath}
\frac{1}{2}\frac{d}{dt}\int_\Omega |v_{+}|^2 + \int_{
  \Omega} \abs{\nabla v_{+}}^2 +  \int_{\partial^R
  \Omega}   \cot(\frac{\pi}{2}\theta) |v_{+}|^{2}  \leq 0 . 
	\end{displaymath}
Therefore, using the continuity up to $t=0$ in  $L^2(\Omega)$, $\int_\Omega |v_{+}(t)|^2 \leq \int_\Omega |v_{+}(0)|^2 =0$ 
and then $v\leq 0$.
\end{proof}

As in the elliptic case, we can also compare solutions with different
types of boundary conditions in the parabolic framework.
\begin{theorem}
	\label{thm:neugeqdir2}
	Let $\Omega\subset\RN$ be a domain with compact boundary and
        let $S^{\theta_1}(t)$ and $S^{\theta_2}(t)$ be  the semigroups
        in Theorem \ref{thm:propiedades_sg_theta} for different $\theta$-boundary
        conditions. Then, for any $1\leq p \leq \infty$ and $0\leq
        u_0\in L^p(\Omega)$ we have 
	\begin{displaymath}
0 \leq \theta_1\leq \theta_2\leq 1 \quad \mbox{implies} \quad   S^{\theta_1}(t)u_0\leq
S^{\theta_2}(t)u_0 \quad t>0 . 
	\end{displaymath}
	In particular, if we denote $k^{\theta_1}$ and $k^{\theta_2}$ the corresponding heat kernels, we have that:
	\begin{displaymath}
	0< k^{\theta_1}(x,y,t)\leq k^{\theta_2}(x,y,t) \qquad
                x,y\in \Omega,  \ t>0. 
	\end{displaymath}
\end{theorem}
\begin{proof}
Assume first $0\leq u_{0}\in L^{2}(\Omega)$ and    consider  $v(t) \defeq S^{\theta_2}(t)u_0-
S^{\theta_1}(t)u_0=u_2(t)-u_1(t)$. From Theorem \ref{thm:neugeqdir}, 
$u_{i}(t)\geq 0$ and from the smoothness in Theorem
\ref{thm:semigrupo_L2}, $v$  satisfies $v_t-\Lap v = 0$ in 
$\Omega\times(0,\infty)$,  $v(0)=0$ in $\Omega$ and
\begin{displaymath}
B_{\theta_{1}}(v(x))=   \sin(\frac{\pi}{2}\theta_1(x))\frac{\partial v}{\partial
                          n}+\cos(\frac{\pi}{2}\theta_1(x))v\geq 0
                        \quad  on \ \partial\Omega\times (0,\infty) . 
\end{displaymath}
This is because $\sin(\frac{\pi}{2}\theta_1) \leq \sin(\frac{\pi}{2}\theta_2)$,
$\cos(\frac{\pi}{2}\theta_1) \geq \cos(\frac{\pi}{2}\theta_2)$  and
$u_2(t)\geq 0$ in $\Omega$,  $\frac{\partial u_2(t)}{\partial n}\leq 0$ on
$\partial \Omega$. This is clear in the Dirichlet and Neumann parts of
the boundary and in the Robin one is because
$0<\theta_{2}<1$. 

Thus,  
$\sin(\frac{\pi}{2}\theta_1)\frac{\partial
  u_2(t)}{\partial n}+\cos(\frac{\pi}{2}\theta_1)u_2(t)\geq
\sin(\frac{\pi}{2}\theta_2)\frac{\partial u_2(t)}{\partial
  n}+\cos(\frac{\pi}{2}\theta_2)u_2(t) = B_{\theta_{2}}(u_{2}(t)) =0$.  
Therefore, we can apply Theorem
\ref{thm:neugeqdir}  to obtain that $v\geq 0$ as claimed (note that, as $v$ is a classical solution of the heat equation, it is in a particular a weak solution, so we can apply Theorem \ref{thm:neugeqdir}).

In particular for any
$0\leq \varphi\in C^\infty_c(\Omega)$:
\begin{displaymath}
 \label{eqn:heatgen4}
 \int_\Omega k^{\theta_1}(x,y,t)\varphi(y)dy=
 S^{\theta_1}(t)\varphi (x) \leq S^{\theta_2}(t)\varphi
 (x) =  \int_\Omega k^{\theta_2}(x,y,t)\varphi(y)dy \ \ \
 \forall x\in \Omega, \ \  \forall t>0, 
\end{displaymath}
and therefore $k^{\theta_1}(x,y,t)\leq k^{\theta_2}(x,y,t)$ for every
$x,y\in \Omega$ and $t>0$.

Finally, if $0\leq u_0\in L^p(\Omega)$, the ordering of the kernels above
and  (\ref{eqn:ackrnpre}) imply  $S^{\theta_1}(t)u_0\leq S^{\theta_2}(t)u_0$.
\end{proof}

As a consequence we get the following important result for exterior domains.

\begin{corollary}
  \label{cor:Gaussian_bound}

Let $\Omega\subset\RN$ be an exterior domain and $k^\theta$ its
associated heat kernel for some homogeneous $\theta-$boundary
conditions. There exists constants $c,C>0$ such that 
\begin{equation}
\label{eqn:gyryabound}
0 < k^\theta(x,y,t) \leq C\frac{e^{-\frac{\abs{x-y}^2}{4ct}}}{t^{N/2}} \qquad
                x,y\in \Omega , \quad t>0 . 
\end{equation}
    
  In particular, 
         \begin{equation}
     \label{eqn:ackrn}
     \norm{S^\theta(t)u_0}_{L^\infty(\Omega)}\leq
     C\frac{\norm{u_0}_{L^1(\Omega)}}{t^{N/2}}  \quad t>0 . 
   \end{equation}
 \end{corollary}
 \begin{proof}
    The Gaussian  bound
(\ref{eqn:gyryabound}) can be found in \cite{gyryathesis}
Theorem 1.3.1 for Neumann boundary conditions (see also \cite{gyrya2011neumann} Theorem 3.10), that is for $\theta
\equiv 1$. Theorem \ref{thm:neugeqdir2}  implies
the bound for other $\theta$-boundary conditions.

The estimate (\ref{eqn:ackrn})  is a  consequence of the
Gaussian bounds above. 
 \end{proof}

\section{The asymptotic  profile}
\label{sec:profiles}

In this section we get back to the case of an exterior domain, that
is, the complement of a compact set $\hole$ that we denote the 
\emph{hole}, which is the closure of a bounded smooth set; hence, 
$\Omega=\RN\backslash \hole$.  
We assume $0\in
\mathring{\hole}$, the interior of the hole,  and observe that   $\hole$ may have different connected
components, although $\Omega$ is connected. 

We are going to construct the asymptotic profile for problem
(\ref{eq:heat_theta}), which is a function that depends only on the
boundary conditions and the domain, which will
characterise the asymptotic mass of solutions as we will see in
Section \ref{sec:asymptoticmass}. We will show below that the profile
can be constructed from parabolic and elliptic arguments and that both
procedures give the same function.

We start with the parabolic construction. 

\begin{lemma}
	\label{lemma:defparabolic}
	Given an exterior domain $\Omega\subset \RN$ and some
homogenous    $\theta$-boundary conditions on $\partial \Omega$, we define
        its associated \textit{$\theta$-parabolic profile} as the
        pointwise monotonically decreasing limit: 
	\begin{equation}
		\label{eqn:defparp}
		\Phi^\theta_p(x)\defeq \lim_{t\to\infty}S^\theta(t)1_\Omega(x) \qquad x\in \Omega,
	\end{equation}
	where $1_\Omega$ is the characteristic function of
        $\Omega$. Then, $\Phi^\theta_p$ is well defined and
        furthermore 
	\begin{displaymath}
		0\leq \Phi^\theta_p(x)\leq 1 \qquad \forall x\in\Omega.
	\end{displaymath}
\end{lemma}
\begin{proof}
From Theorem \ref{thm:propiedades_sg_theta}, the semigroup $S^{\theta}(t)$ is
of contractions in $L^\infty(\Omega)$. Hence in particular,
$\norm{S^{\theta}(t)1_\Omega}_{L^\infty(\Omega)}\leq 1$, that is: 
	\begin{displaymath}
		0\leq S^{\theta}(t)1_\Omega(x)\leq 1_\Omega(x) \qquad \forall x\in \Omega.
	\end{displaymath}
Since the semigroup is order preserving  we obtain:
	\begin{displaymath}
		S^{\theta}(t+s)1_\Omega\leq S^{\theta}(s)1_\Omega \qquad \forall s,t>0,
	\end{displaymath} 
	that is, $S^{\theta}(t)1_\Omega(x)$ is pointwise monotonically
        decreasing in $t$ and  is bounded below by $0$. 
Therefore, the limit in (\ref{eqn:defparp}) is well defined and $0\leq
\Phi^\theta_p\leq 1$. 
\end{proof}

 Now we perform the elliptic construction. Firstly, for every $R>0$ we consider the problem 
\begin{equation}
	\label{eqn:constphie}
	\left\{
	\begin{aligned}
		-\Lap \phi_{R}^\theta(x)&=0 \ \ \ &&\forall x\in \Omega_R\defeq \Omega\cap B(0,R) \\
		B_\theta(\phi_{R}^\theta)(x) &= 0 \ \ \ &&\forall x\in\partial \Omega \\
		\phi_{R}^\theta(x) &= 1 \ \ \ &&\forall \abs{x}=R . 
	\end{aligned}
	\right. 
\end{equation} 

As for large $R$,  $\Omega_R$ is a bounded regular domain, we have a unique solution
$\phi_{R}^\theta$ to this problem (see for example
\cite{gilbarg2015elliptic} Theorem 6.31). Furthermore, we can compare
these functions for different $R$: 
\begin{proposition}
	Let $\Omega\subset \RN$ be an exterior domain, $R_1<R_2$ and
        $\phi^\theta_{{R_1}},\phi^\theta_{{R_2}}$ be defined as in
        (\ref{eqn:constphie}). Then 
	\begin{displaymath}
	1\geq \phi_{{R_1}}^\theta(x)\geq \phi_{{R_2}}^\theta(x)\geq 0 \qquad \forall x\in\Omega_{R_1}.
	\end{displaymath}
\end{proposition}
\begin{proof}
	First of all, as $1_{\Omega_R}$ is a supersolution of (\ref{eqn:constphie}) for any $R>0$,
	using Theorem \ref{thm:neugeqdirell}, we obtain that
        $\phi_{R}(x)\leq 1$ in $\Omega_{R}$.

        Now, for $\abs{x}=R_1$, $\phi_{R_2}(x)\leq 1 =
        \phi_{R_1}(x)$, thus $\phi_{R_2}$ is a subsolution of problem
        (\ref{eqn:constphie}) with $R=R_1$. Hence, using Theorem
        \ref{thm:neugeqdirell} we obtain the result.  
      \end{proof}

With this we construct the  elliptic profile as follows. 
\begin{lemma}
	\label{lemma:defharmonicprof}

	Let the $\theta$-elliptic profile $\Phi_e^\theta$ be defined
        as the monotonically decreasing limit  
	\begin{displaymath}
		\label{eqn:defellp}
0\leq \Phi_e^{\theta}(x)\defeq \lim_{R\to \infty}
\phi^\theta_{R}(x)\leq 1  \qquad x\in \Omega.
	\end{displaymath}

	Then, if $\partial \Omega$ and $\theta$ are regular enough, $\Phi_e^{\theta} $ is a harmonic function
        in $\Omega$, $\Phi_e^\theta \in C^2(\overline{\Omega})\cap
        C^{\infty}(\Omega)$ and
        $B_\theta(\Phi_e^\theta)\equiv 0$ on $\partial \Omega$. In
        particular, $\Phi_e^\theta \in BUC_{\theta}(\Omega)$. 
      \end{lemma}
\begin{proof}
	The limit exists  because of the monotonicity of
        $\phi^{\theta}_{R}(x)$. Furthermore, for any ball $B\subset \Omega$,
        for sufficiently large $R$ we have that $\phi^\theta_{R}$ is
        harmonic in $B$. As the monotonic pointwise limit of harmonic functions
        is harmonic, $\Phi_e^\theta$ is harmonic in $B$. As $B$ was
        arbitrary, $\Phi_e^\theta$ is harmonic in $\Omega$ and hence
        $C^{\infty}(\Omega)$. 
	
	Finally, to see that $\Phi_e^\theta$ satisfies the boundary
        condition, we need to see that the $\{\phi^\theta_{R}\}$ and
        its derivatives converge uniformly,  close to the boundary
        $\partial \Omega$. To do this, we consider, for $R> R_{0}$,
        the restriction of  $\phi^{\theta}_{R}$ to  $\Omega_{R_0}$
        that  satisfies 
	\begin{displaymath}
		\left\{
		\begin{aligned}
			-\Lap \phi^{\theta}_{R}(x) & = 0 \qquad  && \forall x\in \Omega_{R_0} \\
			B_\theta(\phi^{\theta}_{R}) (x) & = 0 \qquad && \forall x\in \partial \Omega \\
			\phi^{\theta}_{R}(x) & \leq 1 \qquad &&
                        \forall \abs{x}=R_0 . 	
		\end{aligned}	
		\right.
	\end{displaymath}
	Then, if $\partial \Omega$ is regular and $\theta \in C^{1+\alpha}(\partial \Omega)$ for $0<\alpha<1$, we apply Schauder estimates from Theorem \ref{thm:ellsch} and obtain
	\begin{displaymath}
		\norm{\phi^{\theta}_{R}}_{C^{2+\alpha}(\adh{\Omega_{R_0}})}\leq C.
	\end{displaymath}

	Hence, using the Ascoli-Arzelà  theorem, we have uniform convergence of a
        subsequence of $\phi_{R}^\theta$ and its derivatives. Therefore,
        as $B_\theta(\phi^\theta_{R})\equiv 0$ in $\partial \Omega$
        for any $R$, then $B_\theta(\Phi_e^\theta)\equiv 0$. 
  In particular, $\Phi_e^\theta$ is in $BUC(\Omega)$ because it is
  continuous up to the boundary and it is a bounded harmonic function,
  so uniformly continuous due to the Schauder estimates of Theorem
  \ref{thm:4.6}.  
      \end{proof}

      \begin{remark}
The values of the function  $\phi^0_{R}(x)$ above is denoted in the
literature as the  harmonic measure for the point $x$ of the set 
        $E=\{y\in \Omega :\abs{y}=R\}$, and is  denoted  as
        $\omega^x_{\Omega_R}(E)$ (see e.g. 
        \cite{harmonicmeasure}).

 Also, the  elliptic profile above when $N\geq 3$ is sometimes referred in the literature as a
  \textit{harmonic profile} or \textit{réduite}. See for example
  \cite{gyrya2011neumann}. Furthermore, it can be understood as
  $1-\omega^x_\Omega(\partial\Omega)$ where $\omega^x_\Omega$ is the
  harmonic measure in $\Omega$ (See \cite{harmonicmeasure}).
\end{remark}

We now  prove that both profiles in fact coincide:
\begin{proposition}
	\label{prop:profilesareequal}
	Let $\Omega=\mathbb{R}^N\backslash \hole$ be an exterior
        domain.
        Then  the   elliptic and parabolic profiles coincide,
        $\Phi_e^\theta=\Phi^\theta_p$,  so we denote them
        $\Phi^{\theta}$. Also $S^\theta(t)\Phi^\theta = \Phi^\theta$ for $t>0$. 
\end{proposition}
\begin{proof}
	$\mathbf{(1). \ \Phi_e^\theta\leq\Phi^\theta_p}:$ 
Since,  $\Phi_e^\theta \in BUC_{\theta}(\Omega)\cap  C^{\infty}(\Omega)$
and harmonic, then it is a strict solution of the heat
        equation (\ref{eq:heat_theta}). Therefore, using the uniqueness of strict solutions
        for $C^0$ semigroups (See for example, \cite{pazy} Theorem
        1.3.) and part (v) in Theorem \ref{thm:propiedades_sg_theta},
        we obtain that $S^\theta(t)\Phi_e^\theta =  \Phi_e^\theta$ for
        $t>0$. 
	
	As $S^\theta(t)$ preserves the order (Theorem
        \ref{thm:propiedades_sg_theta}), and $0\leq \Phi_e^\theta\leq
        1$, we have that 
	\begin{displaymath}
		\Phi^\theta_p=\lim_{t\to\infty}
                S^{\theta}(t)1_\Omega\geq\lim_{t\to\infty} S^{\theta}(t) \Phi_e^\theta =
                \Phi_e^\theta. 
	\end{displaymath}

	$\mathbf{(2). \ \Phi^e_\Omega\geq\Phi^p_\Omega}:$ Consider the
        bounded and uniformly continuous initial data in $\Omega$: 
	\begin{displaymath}
0\leq v_0(x)\defeq \left\{
		\begin{aligned}
			& \phi^\theta_{R}(x) \quad \forall x\in\Omega_R \\
			& 1 \quad \forall x\in\mathbb{R}^N\backslash
                        B(0,R) 
		\end{aligned}
		\right. \leq 1  \quad x\in \Omega, 
	\end{displaymath}
and their  evolution in $\Omega$, $v(t)\defeq S^\theta(t) v_0(x)$,
$t>0$. As $v_0\in BUC_{\theta}(\Omega)$, we have that $v$ is continuous up to
$t=0$ and   $0\leq v(t)\leq 1$ in $\Omega$ because the semigroup is
order preserving and of contractions in $L^\infty(\Omega)$.  Then,
restricted to $\Omega_R$, $v$  satisfies  	
	\begin{displaymath}
		\left\{
		\begin{aligned}
			 v_t-\Lap v & =0 \quad && \forall x\in\Omega_R, \ \ \forall t>0 \\
			 v(x,0) & =\phi^\theta_{R} (x) \quad && \forall x\in\Omega_R, \\
			 B_\theta(v)(x,t) & =0 \quad  && \forall x\in \partial \Omega, \ \ \forall t>0 \\
			 v(x,t) & \leq 1 \quad   && \forall \abs{x}=R, \   \forall t>0 . 
		\end{aligned}
		\right. 
	\end{displaymath}	
	Therefore, if we define $\tilde{v}(t)\defeq
        \restr{v(t)}{\Omega_R}-\phi^\theta_{R}$, it satisfies  in $\Omega_R$:
	\begin{displaymath}
		\left\{
		\begin{aligned}
			 \tilde{v}_t-\Lap \tilde{v} & =0 \quad && \forall x\in\Omega_R, \ \ \forall t>0\\
			 \tilde{v}(x,0) & =0 \quad  && \forall x\in\Omega_R \\
			 B_\theta(\tilde{v})(x,t) & =0 \quad  &&\forall x\in \partial \Omega, \ \ \forall t>0 \\
			 \tilde{v}(x,t) & \leq0 \quad  &&\forall \abs{x}=R,  \ \forall t>0
		\end{aligned}
		\right. 
	\end{displaymath}
	and using Theorem \ref{thm:neugeqdir} in the domain
        $\Omega_{R}$, 
        we obtain that $\tilde{v}\leq 0$,
        i.e. $v(t)\leq \phi^\theta_{R}(x)$ in $\Omega_R$. 
	
Now, let us see that $v(t)$ converges to $\Phi_e^\theta$ when $t\to
\infty$. Consider $w_0=1_\Omega-v_0\geq 0$. Then, $w_0$ has compact
support and then $w_0\in L^1(\Omega)$, so  the estimate
(\ref{eqn:ackrn}), gives  $\lim_{t\to\infty}
\norm{S^\theta(t)w_0}_{L^\infty(\Omega)}=0$.
Then in $\Omega_{R}$ we have 
\begin{displaymath}
0\leq S^\theta(t)1_\Omega(x)=S^\theta(t)(v_0+w_0)(x)\leq \phi^\theta_{R}(x)+
S^\theta(t) w_{0} . 
	\end{displaymath}
Therefore, taking the limit $t\to \infty$ we have 
$\Phi^\theta_p(x) =\lim_{t\to\infty} S^\theta(t)1_\Omega(x) \leq
\phi^\theta_{R}(x)$ in $\Omega_R$ for any $R>0$. Hence, taking $R\to
\infty$, $\Phi^\theta_p(x) \leq \Phi_e^\theta(x)$ as we wanted to
prove. 
\end{proof}

\section{The asymptotic mass of the solutions}
\label{sec:asymptoticmass}

Now we will address the main problem of this paper. Given an initial
datum $u_0\in L^1(\Omega)$, we want to know how much mass is lost
through the hole. First, we define the asymptotic mass of a solution: 
\begin{definition}
  For a given exterior domain $\Omega\subset\RN$ and initial datum
  $u_0\in L^1(\Omega)$ and some homogeneous $\theta$-boundary
  conditions,  we define the asymptotic mass of the solution with
  initial data $u_0$ as
\begin{displaymath}
m^{\theta}_{u_0}\defeq \lim_{t\to\infty}\int_\Omega S^\theta(t)u_0(x)dx.
\end{displaymath}
\end{definition}

Now, we prove a result  that characterises  the asymptotic mass of a
solution in terms of the initial datum and the asymptotic profile: 
\begin{proposition}
\label{prop:parabolicremaining}

Let $\Omega\subset \RN$ an exterior domain and $ u_0\in L^1(\Omega)$
an initial datum. Then, if we denote $u(t)= S^{\theta}(t)u_{0}$ the solution of the heat
equation (\ref{eq:heat_theta})  in $\Omega$ with homogenous
$\theta$-boundary conditions  
and the initial datum $u_0$, we have that
\begin{equation}
\label{eqn:asymptotic_mass}
m^{\theta}_{u_0} = \int_\Omega u_0(x) \Phi^\theta                (x)dx 	
\end{equation}
where $\Phi^\theta$ is the asymptotic profile in  $\Omega$.
In particular, $\D \int_\Omega u(t)\Phi^\theta$ is constant in time. 
\end{proposition}
\begin{proof}
  With the notations above, we have, 
\begin{displaymath}	
\label{eqn:asymptmass}	
m^{\theta}_{u_0}=\lim_{t\to \infty} \int_\Omega S^{\theta}(t)u_0(x)
1_\Omega(x)dx\myeq{(\ref{eqn:Sexchange})}\lim_{t\to \infty}
\int_\Omega u_0(x)
S^{\theta}(t)1_\Omega(x)dx\myeq{(\ref{eqn:defparp})} 
\int_\Omega u_0(x) \Phi^\theta (x)dx, 
\end{displaymath}
where we have used the dominated convergence theorem in the last step
as $\abs{u_0S^{\theta}(t)1_\Omega}\leq \abs{u_0}$ which is integrable. 
		
To see that $\D \int_\Omega u(t)\Phi^\theta$ is a
conserved quantity, we just use that, for $t, s>0$, 
\begin{displaymath}
\int_\Omega u(t+s) \Phi^\theta = \int_\Omega  S^\theta(s)u(t) 
\Phi^\theta   \myeq{(\ref{eqn:Sexchange})}  \int_\Omega   u(t) S^\theta(s)\Phi^{\theta}=
\int_\Omega u(t) \Phi^{\theta}
\end{displaymath}
\end{proof}

\begin{remark}
The profile  $\Phi^\theta$ provides an explicit computation of the
amount of mass lost for any solution. It also has an interesting
interpretation in terms of initial point masses. Actually, since, 
  \begin{displaymath}
    S^{\theta}(t)1_\Omega(x)=\int_\Omega k^\theta(x,y,t)1_\Omega(y)dy , 
  \end{displaymath}
 using that  $k^\theta$ is spatially symmetric, see Theorem
 \ref{thm:propiedades_sg_theta}, 
  \begin{displaymath}
    S^{\theta}(t)1_\Omega(x)=\int_\Omega k^\theta(y,x,t)dy=\norm{S^{\theta}(t)\delta(\cdot-x)}_{L^1(\Omega)},
  \end{displaymath}
  where $\delta$ is the Dirac distribution. Therefore,
  \begin{displaymath}
    \Phi^\theta(x)=\lim_{t\to\infty}S^{\theta}(t)1_\Omega(x)=\lim_{t\to\infty}\norm{S^{\theta}(t)\delta(\cdot-x)}_{L^1(\Omega)},
  \end{displaymath}
  which is the remaining mass for a point source in $x$ as initial
  condition. Therefore (\ref{eqn:asymptotic_mass}) reflects the
  contribution of a mass density $u_{0}(x)$ at each $x\in \Omega$ in the
  total remaining mass. 
  \begin{reserva2}
    This can be done rigorously extending the definition of the
    semigroup to totally bounded variation measures (See for example
    \cite{aroberrobinson1}). However, this is something we will not
    need through the paper and we omit it to not to extend so much.
  \end{reserva2}
\end{remark}

Now we give some estimates on the profile $\Phi^{\theta}$. In
particular, we prove that for $N=1, 2$, $\Phi^{\theta} \equiv 0$ while if
$N\geq 3$ then $\Phi^{\theta} \not \equiv 0$ and it actually converge to
$1$ as $|x| \to \infty$. Therefore in two dimension all the mass of
every solution is lost through the hole, while for higher dimensions
there is always some remaining mass.

First, the following result  allows us to compare the asymptotic
profiles for different boundary conditions: 
\begin{proposition}
\label{prop:compasym}

Let $\Omega\subset\RN$ be an exterior domain and $\Phi^{\theta_1}$,
$\Phi^{\theta_2}$ asymptotic profiles in  $\Omega$ for different
$\theta$-boundary conditions. Then, if $0\leq \theta_1\leq
\theta_2\leq 1$ we have that
\begin{displaymath}
\Phi^{\theta_1}\leq \Phi^{\theta_2} \quad \mbox{in $\Omega$}. 
\end{displaymath}

In particular, if $0\leq u_{0} \in L^{1}(\Omega)$ and $0\leq \theta_1\leq
\theta_2\leq 1$ the asymptotic masses satisfy
\begin{displaymath}
  m^{\theta_{1}}_{u_0} \leq m^{\theta_{2}}_{u_0} . 
\end{displaymath}

\end{proposition}
\begin{proof}
This  is just a consequence of Theorem \ref{thm:neugeqdir2}: 
\begin{displaymath}
\Phi^{\theta_1} \myeq{(\ref{eqn:defparp})} \lim_{t\to\infty}
\int_\Omega S^{\theta_1}(t)1_\Omega \myleq{Thm \ref{thm:neugeqdir2}}
\lim_{t\to\infty} \int_\Omega S^{\theta_2}(t)1_\Omega
\myeq{(\ref{eqn:defparp})} \Phi^{\theta_2}.
\end{displaymath}

The rest is immediate. 
\end{proof}

In the same way, if instead of changing the boundary conditions, we
change the domain, for homogeneous Dirichlet boundary conditions we
have the following comparison result.

\begin{proposition}
\label{prop:compellprof}
Let $\Omega_1 \subset \Omega_2\subset\RN$ be two exterior
domains and consider Dirichlet boundary conditions, that is $\theta
=0$, in both of them.  Then we have that their  asymptotic
profiles satisfy:
\begin{displaymath}
\Phi^{0}_{\Omega_1}(x)\leq \Phi^{0}_{\Omega_2}(x) \qquad \forall x\in \Omega_1.
\end{displaymath}
\end{proposition}
\begin{proof}
With the  notation  in (\ref{eqn:constphie}), we have that
  $\phi^0_{(\Omega_2)_R}\geq 0$ in $\partial \Omega_1$, while 
  $\phi^0_{(\Omega_1)_R}= 0$ on $\partial \Omega_1$.  Therefore, if we
  denote $v\defeq\phi^0_{(\Omega_1)_R}-\phi^0_{(\Omega_2)_R}$ we have
  that:
\begin{displaymath}
\left\{
\begin{aligned}
-\Lap & v(x)=0 \ \ \ \forall x\in (\Omega_1)_R \\
& v(x) \leq 0 \ \ \ \forall x\in \partial(\Omega_1)_R . 
\end{aligned}
\right. 
\end{displaymath}
Then, using Theorem \ref{thm:neugeqdirell}, we obtain that $v\leq 0$ in
$(\Omega_1)_R$, or equivalently
$\phi^0_{(\Omega_1)_R}\leq\phi^0_{(\Omega_2)_R}$ in
$(\Omega_1)_R$. Taking the limit when $R\to \infty$ we have
$\Phi^0_{\Omega_1}\leq \Phi^0_{\Omega_2}$ in $\Omega_1$.
\end{proof}

Now, we will explore briefly what the form of the asymptotic profile
$\Phi^\theta$ is. For  this, we will firstly study the case when the
domain is for the complement of a ball
$D_r=\mathbb{R}^N\backslash B(0,r)$ and  $\theta$ is a constant.

\begin{lemma}
\label{lemma:prelimit}

Let $D_r^R =B(0,R)\backslash B(0,r)\subset \mathbb{R}^N$ and
$0\leq \theta <1$ a constant. Then the  solution to the problem
\begin{displaymath}
\label{eqn:lemmaeq1}
\left\{
\begin{aligned}
- \Lap  \phi (x) & =0 \ \ \ &&\forall x\in \Omega_R \\
B_\theta(\phi)(x) & = 0 \ \ \ &&\forall \abs{x}=r \\
\phi(x) & = 1 \ \ \ &&\forall \abs{x}=R \\
\end{aligned}
\right.
\end{displaymath}
is given by 
\begin{displaymath}
 \Phi_{r,R}(x)=\left\{
\begin{array}{ll}
\frac{\abs{x}-C_\theta r}{R-C_\theta r} & N=1 \\
\frac{\log(\abs{x})-C_\theta\log(r)}{\log(R) -C_\theta\log(r)} & N=2 \\
\frac{\abs{x}^{2-N} - C_\theta  r^{2-N}}{R^{2-N}-C_\theta r^{2-N}} &  N\geq 3 , 
\end{array} 
\right. 
\end{displaymath}
where 
\begin{equation*}
C_\theta \defeq \left\{
\begin{array}{ll}
1+\tan(\pi\theta/2)  & N\leq 2 \\
1+(N-2)\tan(\pi\theta/2) & N\geq 3 . 
\end{array} 
\right. 
\end{equation*}

If $\theta=1$, then $\Phi_{r,R}\equiv 1$.
\end{lemma}
\begin{proof}
A radial harmonic  function satisfies 
  \begin{displaymath}
\frac{1}{r^{N-1}}\frac{\partial}{\partial r}\left(r^{N-1}\frac{\partial f}{\partial r}\right)=0,
\end{displaymath}
and then  when $N\geq 3$,  $f(r)=\frac{A}{r^{N-2}}+B$. Choosing the
constants to fit the boundary conditions we get the result. The cases
$N=1,2$ follow with slight changes. 
\end{proof}

Now, letting $R\to \infty$,  we get the explicit form of the asymptotic profile. 
\begin{lemma}
  \label{lemma:ellball}
  
Let $0\leq \theta\leq 1$ be a constant. The asymptotic profile for
the exterior domain $D_r\defeq \RN\backslash \adh{B(0,r)}$ with
$\theta$-boundary conditions  has the following
explicit form: 
  \begin{displaymath}
\Phi^{\theta}(x)=\left\{
\begin{array}{lll}
0 & \theta\neq 1 \ \ N\leq 2  \\
1-\frac{r^{N-2}}{C_\theta\abs{x}^{N-2}} & \theta\neq 1  \ \ N\geq 3 \\
1 & \theta= 1 \ \mbox{in any dimension}
\end{array} 
\right. \quad x \in D_r\defeq \RN\backslash \adh{B(0,r)}
\end{displaymath}
	where $C_\theta$ is as in Lemma \ref{lemma:prelimit}.
\end{lemma}
As we see, if $N\leq 2$, the asymptotic profile is the zero function
for Dirichlet or Robin conditions. Furthermore, for $N\geq 3$ the
asymptotic profile tends to $1$ when $\abs{x}\to \infty$. This also
happens for an arbitrary exterior domain and general $\theta$-boundary
conditions, as we will show below. 

First, for a general exterior domain, we analyse the particular case
of homogeneous Dirichlet boundary conditions: 
\begin{proposition}
\label{prop:dirasymppro}
Let $\Omega =\mathbb{R}^N\backslash \hole$ be an exterior domain with
Dirichlet boundary conditions, that is, $\theta \equiv 0$.   Then
the asymptotic  profile satisfies:
\begin{enumerate}
\item
  $\Phi^0=0$ if $N\leq 2$.

\item
  $\Phi^0\geq 0$  and $\Phi^{0}(x)\to 1$ when
  $\abs{x}\to \infty$ if $N\geq 3$. In fact, there exists a constant
  $C>0$ depending on $\Omega$ such that:
\begin{displaymath}
1-\frac{C}{\abs{x}^{N-2}}\leq \Phi^{0}(x) \leq 1 \qquad \forall x\in \Omega.
\end{displaymath}
\end{enumerate}
\end{proposition}
\begin{proof}
 Let us take  $r>0$ such that $B(0,r)\subset \hole$, which implies  $\Omega
 \subset D_r=\RN\backslash B(0,r)$. Then, using Proposition 
 \ref{prop:compellprof}, we obtain that $\Phi^0_{\Omega}\leq
 \Phi^0_{D_r}$ in $\Omega$. This  automatically implies that, for
 $N\leq 2$, $\Phi^0_{\Omega}\equiv  0$.

 Now for $N\geq 3$  choose $r>0$ such that $\hole \subset B(0,r)$, which implies 
 $D_r\subset \Omega$. Then Proposition \ref{prop:compellprof} gives now 
 $\Phi^0_{D_r} \leq \Phi^0_{\Omega}\leq 1$ in $D_r$ and by 
 Lemma \ref{lemma:ellball} we get the estimate in $D_{r}$. Taking a
 larger constant, we get the estimate in $\Omega$. 
 \end{proof}

Now for other $\theta$-boundary conditions we have the following. 

\begin{theorem}
\label{thm:compperf}

Assume $0\leq \theta \leq 1$, not necessarily constant. Then, if $\partial \Omega$ and $\theta$ are sufficiently regular, the asymptotic profile
$\Phi^{\theta}$ satisfies: 
\begin{enumerate}
\item
  If $N\geq 3$,
\begin{displaymath}
1-\frac{C}{\abs{x}^{N-2}}\leq \Phi^{\theta}(x) \leq 1 \qquad \forall x\in \Omega.
\end{displaymath}

\item
  If $N\leq 2$ and $\theta\not \equiv 1$ then $\Phi^{\theta}\equiv 0$.
\end{enumerate}

If $\theta\equiv 1$, that is, for Neumann boundary conditions,  then
$\Phi^{1}\equiv 1$ in any dimensions. 

\end{theorem}
\begin{proof}
For (i)  it is enough to use Proposition 
\ref{prop:dirasymppro} jointly with Proposition \ref{prop:compasym}. 

For (ii)  take $M=\max_{\partial\Omega} \Phi^\theta$
and  $u\defeq \Phi^\theta-M$. Then, for $R>0$ large enough, $u$
satisfies 
\begin{displaymath}
\left\{
\begin{aligned}
 -\Lap  & u = 0 \qquad in \ \Omega_R \\ 
& u\leq 0 \qquad  in \ \partial \Omega \\
& u\leq 1 \qquad  in \ \partial B(0,R) . 
\end{aligned}
\right. 
\end{displaymath}
Therefore, $u$ is a subsolution of (\ref{eqn:constphie}) for
$\theta=0$, that is, for the Dirichlet problem. Therefore, using
Theorem \ref{thm:neugeqdirell}, we obtain that: 
\begin{displaymath}
u(x)\leq \phi^0_{R}(x) \qquad \forall x\in \Omega_R
\end{displaymath}
and since  $R>0$ is arbitrary, we get 
\begin{displaymath}
u(x)\leq \Phi^0(x) \qquad x\in\Omega.
\end{displaymath}
Since  $N\leq 2$, from  Proposition 
\ref{prop:dirasymppro}, $\Phi^0\equiv 0$ and  therefore $0\leq
\Phi^\theta\leq M$ in $\Omega$.

But then, at a point  $x_0\in \partial\Omega$ of maximum, that is,  such that $\Phi^\theta(x_0)=M$, 
by Hopf Lemma we have  $\frac{\partial \Phi_e^\theta}{\partial
  n}(x_0)> 0$ or $\Phi^\theta$ is constant. In the latter case, since
$\theta \not \equiv 1$, we have  $\Phi^\theta\equiv 0$. In the former case,
if  $x_{0}\in \partial^{D} \Omega$ then $M=0$ and
$\Phi^\theta\equiv 0$. If $x_{0} \in \partial^{R} \Omega$  
\begin{displaymath}
0=B_\theta(\Phi^\theta)(x_0)=\sin(\frac{\pi}{2}\theta(x_0))\frac{\partial \Phi^\theta}{\partial
  n}(x_0)+\cos(\frac{\pi}{2}\theta(x_0))M \geq
\cos(\frac{\pi}{2}\theta(x_0))M 
\end{displaymath}
and then $M=0$ and thus  $\Phi^\theta \equiv 0$. 
\begin{reserva}
  Now, if $\cos(\frac{\pi}{2}\theta(x_0))> 0$, this implies $M=0$ and
  thus $\Phi^\theta \equiv 0$. In the case
  $\cos(\frac{\pi}{2}\theta(x_0))=0$, we have
  $\frac{\partial \Phi^\theta}{\partial n}(x_0)=0$. Therefore, using
  Hopf lemma, we obtain that $\Phi^\theta\equiv M$. As
  $\theta \not \equiv 1$, we have $M=0$ so $\Phi^\theta\equiv 0$.
\end{reserva}
\end{proof}

Then we have the following
result about the rate of mass loss. Observe that, except for Neumann
boundary conditions,  for $N\geq 3$ all
solutions lose mass at a uniform rate, while if  $N\leq 2$ there are
solutions for which the mass  decays as slow as we want. 

\begin{theorem}
  \label{thm:rate_loss_mass}
  
  Assume $\Omega$ is an exterior domain and  $0\leq \theta \leq
  1$,  not necessarily constant, with $\theta\not\equiv 1$, that is, except for Neumann boundary
  conditions. For every  $0\leq u_{0} \in L^{1}(\Omega)$ denote 
\begin{displaymath}
m^{\theta}_{u_0}(t) \defeq \int_\Omega S^\theta(t)u_0(x)dx , \quad t>0 
\end{displaymath}
and the asymptotic mass  $m^{\theta}_{u_0}\defeq
\lim_{t\to\infty} m^{\theta}_{u_0}(t)$. Then 

  \begin{enumerate}
  \item
If    $N\geq 3$, there exists a constant $c_{\theta}$ such that 
\begin{displaymath}
\big|  m^{\theta}_{u_0}(t)  -m^{\theta}_{u_0}\big| \leq c_{\theta}
\frac{\|u_{0}\|_{L^{1}(\Omega)}}{t^{\frac{N-2}{2}}}. 
\end{displaymath}

\item 
If  $N\leq 2$,  let  $g:[0,\infty) \to (0,1]$ a monotonically decreasing continuous
function such that $\lim_{t\to\infty}g(t)=0$.

Then, there exist an initial
value $0\leq u_0\in L^1(\Omega)$ with $\norm{u_0}_{L^1(\Omega)}=1$ and $T>0$ such that
\begin{equation}
\label{eqn:opteq1bis}
m^{\theta}_{u_0}(t) \geq g(t)  \qquad \forall t\geq T.
\end{equation}
\end{enumerate}
\end{theorem}
\begin{proof}
(i) As we  have $m^{\theta}_{u_0} = \int_{\Omega} u_{0} \Phi^{\theta}$ then
  \begin{displaymath}
     m^{\theta}_{u_0}(t)  -m^{\theta}_{u_0} = \int_{\Omega} u_{0}
     \big( S^\theta(t)1_{\Omega} -\Phi^{\theta}\big) = \int_{\Omega} u_{0}
    S^\theta(t)  \big( 1_{\Omega} -\Phi^{\theta}\big) 
  \end{displaymath}
  since $ S^\theta(t) \Phi^{\theta} =    \Phi^{\theta}$.
  From Theorem \ref{thm:compperf}, since $N\geq 3$,
  $0\leq 1_{\Omega} (x)-\Phi^{\theta}(x) \leq \frac{C}{\abs{x}^{N-2}}$
  in $\Omega$ and then $1_{\Omega} -\Phi^{\theta}\in L^{p}(\Omega)$
  for any $p> \frac{N}{N-2}$. Then the Gaussian bounds
  (\ref{eqn:gyryabound})
  \begin{reserva2}
    and the standard estimate using Young's inequality give
    \begin{displaymath}
      \| S^\theta(t)  \big( 1_{\Omega} -\Phi^{\theta}\big)
      \|_{L^{\infty}(\Omega)} \leq \frac{c}{t^{\frac{N}{2p}}}
      \|1_{\Omega} -\Phi^{\theta} \|_{L^{p}(\Omega)} . 
    \end{displaymath}
    Thus we get the result with
    $\alpha= \frac{N}{2p} < \frac{N-2}{2}$.
  \end{reserva2}
imply 
    \begin{displaymath}
       S^\theta(t)  \big( 1_{\Omega} -\Phi^{\theta}\big) (x) \leq
   C    \int_{\Omega} \frac{e^{-\frac{\abs{x-y}^2}{4ct}}}{t^{N/2}}
   \frac{1}{\abs{y}^{N-2}}\, dy =
   \begin{Bmatrix}
     z= \frac{x-y}{\sqrt{t}} \\
     y= x-\sqrt{t}z
   \end{Bmatrix}
   \leq C    \int_{\R^{N}} e^{-\frac{\abs{z}^2}{4c}} 
   \frac{1}{\abs{x-\sqrt{t}z}^{N-2}}\, dz 
    \end{displaymath}
    \begin{displaymath}
      = \frac{C}{t^{\frac{N-2}{2}}}   \int_{\R^{N}} e^{-\frac{\abs{z}^2}{4c}} 
   \frac{1}{\abs{\frac{x}{\sqrt{t}}-z}^{N-2}}\, dz  =
   \frac{C}{t^{\frac{N-2}{2}}}  F(\frac{x}{\sqrt{t}})
    \end{displaymath}
    with
    \begin{displaymath}
      F(w) = \int_{\R^{N}} e^{-\frac{\abs{z}^2}{4c}} 
   \frac{1}{\abs{w-z}^{N-2}}\, dz   = \int_{|w-z|\leq a} e^{-\frac{\abs{z}^2}{4c}} 
   \frac{1}{\abs{w-z}^{N-2}}\, dz  + \int_{|w-z|\geq a} e^{-\frac{\abs{z}^2}{4c}} 
   \frac{1}{\abs{w-z}^{N-2}}\, dz  
    \end{displaymath}
    for any $a>0$. Now
    \begin{displaymath}
      F_{1}(w) = \int_{|w-z|\geq a} e^{-\frac{\abs{z}^2}{4c}} 
   \frac{1}{\abs{w-z}^{N-2}}\, dz  \leq
   \frac{1}{a^{N-2}}\int_{|w-z|\geq a} e^{-\frac{\abs{z}^2}{4c}} 
  \, dz \leq   
   \frac{1}{a^{N-2}}\int_{\R^{N}} e^{-\frac{\abs{z}^2}{4c}}  = C 
    \end{displaymath}
    and
    \begin{displaymath}
       F_{2}(w) = \int_{|w-z|\leq a} e^{-\frac{\abs{z}^2}{4c}} 
   \frac{1}{\abs{w-z}^{N-2}}\, dz  = \int_{\R^{N}} e^{-\frac{\abs{z}^2}{4c}} 
   \frac{1}{\abs{w-z}^{N-2}}\chi_{B(w,a)} \, dz . 
    \end{displaymath}
The function $g(z)=  \frac{1}{\abs{w-z}^{N-2}}\chi_{B(w,a)} $ is in
$L^{q}(\R^{N})$ for $q<\frac{N}{N-2}$ with a norm independent of $w$ so Hölder's inequality
implies
\begin{displaymath}
   F_{2}(w) \leq C. 
\end{displaymath}
Hence $F\in L^{\infty}(\R^{N})$ and therefore 
 \begin{displaymath}
    \| S^\theta(t)  \big( 1_{\Omega} -\Phi^{\theta}\big)
    \|_{L^{\infty}(\Omega)} \leq  \frac{c}{t^{\frac{N-2}{2}}}  
  \end{displaymath}
and we get the result. 
\begin{reserva}
	Una demostración más corta y creo que cierta. Queremos estudiar:
	\begin{equation}
		\int_\RN \frac{e^{-\frac{\abs{x-y}^2}{4ct}}}{t^{N/2}}\frac{1}{\abs{y}^{N-2}}
	\end{equation}
Entonces tomemos $f(y)=e^{-\frac{\abs{x-y}^2}{4ct}}$ y $g(y)=\abs{y}^{2-N}$. Consideramos sus reordenadas decrecientes que son $f^*(y)=e^{-\frac{\abs{y}^2}{4ct}}$ y $g^*(y)=\abs{y}^{2-N}$ (deberían ser esas, pero igual no, me estoy tirando un triple). La desigualdad de Hardy-Littlewood nos dice:
		\begin{equation}
			\int_\RN f(y)g(y)\leq \int_\RN f^*(y)g^*(y)
		\end{equation}
	Entonces
	\begin{equation}
		\frac{1}{t^{N/2}}\int_\RN \frac{e^{-\frac{\abs{x-y}^2}{4ct}}}{\abs{y}^{N-2}} \leq \frac{1}{t^{N/2}}\int_\RN \frac{e^{-\frac{\abs{y}^2}{4ct}}}{\abs{y}^{N-2}} =\frac{C}{t^{N/2}}\int_0^\infty e^{-\frac{r^2}{4ct}}r=\frac{Ct}{t^{N/2}}=\frac{C}{t^{\frac{N-2}{2}}}
	\end{equation}
\end{reserva}
  
  \noindent (ii)
  Observe that we have $\Phi^{\theta}=0$ and then   $m^{\theta}_{u_0}
  = 0$, for any initial data in $L^{1}(\Omega)$. Now  consider
  $\lambda>0$ the first Dirichlet eigenvalue for the 
Laplacian operator in the unit ball, $B$,  and its associated positive
eigenfunction $\psi\geq 0$ (normalized such that
$\norm{\psi}_{L^1(B)}=1$). 
	
Now we choose $t_n\to \infty$ such that
$g(t_n)=\frac{1}{2^{n+2}}$. Then we consider the following initial
datum made up by rescaled copies of $\psi$ in disjoint balls with
large radius and far away centres: 
\begin{displaymath}
	u_0(x)=\sum_{n=1}^\infty\frac{1}{2^n R_n^{N}} \psi\left(\frac{x-x_n}{R_n}\right)\chi_{B(x_n,R_n)}(x),
\end{displaymath}
where $\chi_{B(x_n,R_n)}$ is the characteristic function of the $B(x_n,R_n)$ and $R_n$ and $x_n$ are chosen so that
\begin{enumerate} [label={(\arabic*)}]
\item $e^{-\frac{\lambda}{R_n^2}t_n}\geq 1/2$ (This is  possible
  taking $R_n$ large enough)
  
\item $B(x_n,R_n)\subset \Omega$ (This is possible  taking $\abs{x_n}$ large enough)
\end{enumerate}
Therefore, $\norm{u_0}_{L^1(\Omega)}=1$ and
\begin{equation}
\label{eqn:opteq1}
\int_\Omega S^\theta (t_n)u_0(x)dx  \mygeq{Thm
  \ref{thm:neugeqdir2}} \int_\Omega S^0(t_n)u_0 (x)dx\mygeq{}
\frac{1}{2^n R_n^{N}}   \int_\Omega S^0(t_n)\left(
  \psi\left(\frac{\cdot-x_n}{R_n}\right)\chi_{B(x_n,R_n)}\right)(x)dx . 
\end{equation}
Now observe that for  $0\leq \varphi$ in $\Omega$, as $B(x_n,R_n)\subset
\Omega$, we have $S^0(t)\varphi\geq 0$ in $\partial B(x_n,R_n)$ and
then Theorem \ref{thm:neugeqdir} implies 
$S^0(t)\varphi \geq S^0_{B(x_n,R_n)}\varphi$, that is, the heat
semigroup in $B(x_n,R_n)$ with Dirichlet boundary conditions.
Therefore, 
\begin{equation} \label{eqn:opteq2}
\begin{aligned} 
 \frac{1}{2^n R_n^{N}} &\int_\Omega S^0(t_n) \left(\psi\left(\frac{\cdot-x_n}{R_n}\right)\chi_{B(x_n,R_n)}\right)(x)dx
 \geq  \frac{1}{2^n R_n^{N}} \int_{B(x_n,R_n)}
 S^0_{B(x_n,R_n)}(t_n)\psi\left(\frac{\cdot-x_n}{R_n}\right) (x)dx \\
& = \frac{e^{-\frac{\lambda}{R_n^2} t_n}}{2^{n}R_n^N}
\int_{B(0,R_n)} \psi(x) dx\mygeq{(1)} \frac{1}{2^{n+1}}. 
\end{aligned}	
\end{equation}	
Combining (\ref{eqn:opteq1}) and (\ref{eqn:opteq2}) we obtain 
\begin{equation}
	\label{eqn:opteq3}
	m_{u_0}^\theta(t_n)= \int_\Omega S^\theta (t_n)u_0(x)dx \geq
        \frac{1}{2^{n+1}} .
\end{equation}
Now, take $T=t_1$ and $t\geq T$. Then, there exists $n\in\mathbb{N}$ such that $t\in [t_n,t_{n+1})$. As $g$ and the mass of the solutions are monotonically decreasing we obtain
\begin{displaymath}
	m^\theta_{u_0}(t)\geq m^\theta_{u_0}(t_{n+1}) \mygeq{(\ref{eqn:opteq3})}\frac{1}{2^{n+2}}=g(t_{n})\geq
	g(t)
\end{displaymath}
 which is (\ref{eqn:opteq1bis}).
\end{proof}

\appendix

\begin{appendices}
\section{Schauder estimates}
\label{sec:schauder-estimates}

Here we present some elliptic Schauder estimates, which allow us to
estimate the derivatives of a solution of the Laplace equation just
with the $L^\infty$ norm of the solutions. These are classical results
which can be found, for example, in \cite{gilbarg2015elliptic} Theorem
4.6:
\begin{theorem}
  \label{thm:4.6}
  Let $\Omega\subset\RN$ and $u\in C^2(\Omega)$ such that
  $$\Lap u(x) =0 \ \ \ \forall x\in\Omega.$$ Then, for $x_0\in \Omega$
  and any two concentric balls $B_1\defeq B(x_0,R)$ and
  $B_{2}\defeq B(x_0, 2R)\subset \subset \Omega$, we have
  \begin{displaymath}
    \label{eqn:thm4.6}
    R\abs{Du}_{B_1}+R^2\abs{D^2u}_{B_1}\leq C\norm{u}_{L^\infty(B_2)},
  \end{displaymath}
  where we denote $\abs{Du}_{B_1}=\max_{i}\norm{D_i
    u}_{L^\infty(B_1)}$, $\abs{D^2u}_{B_1}=\max_{i,j}\norm{D_{ij}
    u}_{L^\infty(B_1)}$ and $C$ is a constant independent on $u$,
  $x_0$ and $R$. 
\end{theorem}
In the case of homogeneous boundary conditions we also have Schauder
estimates. The following result can be found in
\cite{gilbarg2015elliptic} Theorem 6.30 for Robin and Neumann boundary
conditions, and in \cite{gilbarg2015elliptic} Theorem 6.6 for
Dirichlet boundary conditions.
\begin{theorem}
  \label{thm:ellsch}
  Let $\Omega\subset \RN$ be a bounded regular domain. Let
  $u\in C^{2+\alpha}(\Omega)$ be a function such that:
  \begin{displaymath}
    \left\{
      \begin{aligned}
	-\Lap u &= f \qquad && \Omega \\
	B_\theta(u)&=\psi \qquad && \partial \Omega . 
      \end{aligned}
    \right. 
  \end{displaymath}
  with $\theta\in C^{1+\alpha}(\partial \Omega)$. Then,
  \begin{displaymath}
    \norm{u}_{C^{2+\alpha}(\adh{\Omega})}\leq C
    (\norm{u}_{C^0(\adh{\Omega})}+\norm{f}_{C^{\alpha}(\adh{\Omega})}+\norm{\psi}_{C^{1+\alpha}(\partial\Omega)}) 
  \end{displaymath}
where $\norm{\psi}_{C^{1+\alpha}(\partial\Omega)}=\inf\{\norm{\varphi}_{C^{1+\alpha}(\adh{\Omega})} : \varphi\in C^{1+\alpha}(\adh{\Omega}), \ \varphi\equiv \psi \ on \ \partial \Omega\}$.
\end{theorem}
\end{appendices}

\bibliographystyle{alpha}
\bibliography{extdom}
\end{document}